\documentclass[a4paper]{amsart}

\usepackage{amsmath,amsthm,amsfonts, amssymb}
\usepackage{tikz, tikz-cd}
\usepackage{setspace}
\usepackage{mathtools} 
\usepackage{times}
\usepackage{amsrefs}
\usepackage[T1]{fontenc}
\usepackage{hyperref} 

\numberwithin{equation}{section}
\theoremstyle{plain}
\newtheorem{theorem}[equation]{Theorem}
\newtheorem*{theorem*}{Theorem}
\newtheorem{lemma}[equation]{Lemma}
\newtheorem{proposition}[equation]{Proposition}

\newtheorem{corollary}[equation]{Corollary}

\theoremstyle{definition}
\newtheorem{example}[equation]{Example}
\newtheorem{definition}[equation]{Definition}
\newcommand{\cg}[1]{\langle #1 \rangle}
\newcommand{\C}{\mathbb{C}}             
\renewcommand{\H}{\mathcal{H}}
\newcommand{\id}{\operatorname{id}}        
\newcommand{\Q}{\mathbb{Q}}                 
\newcommand{\R}{\mathbb{R}}                
\newcommand{\Z}{\mathbb{Z}}                
\DeclareMathOperator{\twist}{Twist}
\DeclareMathOperator{\coker}{coker}
\newcommand{\pr}{\textnormal{pr}}
\DeclareMathOperator{\Ind}{Ind}
\DeclareMathOperator{\cyl}{cyl}

\newcommand{\sbm}[1]{{\let\amp=&\left(\begin{smallmatrix}#1\end{smallmatrix}\right)}} 

\usetikzlibrary{arrows, backgrounds}

\definecolor{mypink}{RGB}{220, 20, 95}
\definecolor{myblue}{RGB}{55, 155, 215}
\definecolor{mygreen}{RGB}{35, 175, 145}
\definecolor{mydarkblue}{RGB}{70, 80, 170}

\title{Equivariant Topological T-Duality}
\author{Tom Dove}
\author{Thomas Schick}

\date{\today}

\address{Mathematisches Institut, Universit\"at G\"ottingen, Bunsenstr. 3-5, 37073 G\"ottingen, Germany}

\email{tomjdove@gmail.com}
\email{thomas.schick@math.uni-goettingen.de}

\begin{document}

\onehalfspacing

\begin{abstract}
Topological T-duality is a relationship between pairs $(E,P)$ over a fixed space $X$, where $E \to X$ is a principal torus bundle and $P \to E$ is a twist, such as a gerbe of principal $PU(\H)$-bundle.
This is of interest to topologists because of the T-duality transformation: a T-duality relation between pairs $(E, P)$ and $(\hat E, \hat P)$ comes with an isomorphism (with degree shift) between the twisted K-theory of $E$ and the twisted K-theory of $\hat E$.
We formulate topological T-duality for circle bundles in the equivariant setting, following the definition of Bunke, Rumpf, and Schick.
We define the T-duality transformation in equivariant K-theory and show that
it is an isomorphism for all compact Lie groups, equal to its own inverse and
uniquely characterized by naturality and a normalization for trivial situations.
\end{abstract}

\maketitle
\tableofcontents

\section{Introduction}

Topological T-duality over a base space $X$ is a relation between pairs $(E,P)$ consisting of a $T^n$-bundle $E \to X$ and a twist $P \to E$.
A twist in this context is the same as in twisted equivariant K-theory; we, for instance, consider principal $PU(\mathcal H)$-bundles.
This models the underlying topology of T-duality in string theory, where it is a duality between models of space-time.
In this context, $E$ is the ``background'' in which strings propagate and $P$ is the ``H-flux'' describing the B-field.
The topological picture is necessarily not the full story, since we are forgetting a significant amount of geometry.
Nonetheless, one can get valuable insights into the global topology of the string background and its T-dual \cite{BEM}.
Such insights are difficult to make when performing calculations locally, as physicists often do.

In topological T-duality, two pairs $(E, P)$ and $(\hat E, \hat P)$ are T-dual if there is an isomorphism between the pullbacks of $P$ and $\hat P$ to $E \times_X \hat E$ that satisfies a certain ``Poincar\`e bundle'' condition \cite{BunkeRumpfSchick}.
\[
\begin{tikzcd}[column sep={4em,between origins},row sep=1em]
    & p^*P \arrow[ld] \arrow[rd] \arrow[rr, "\cong", swap] \arrow[rr, "u"] & & \hat p^* \hat P \arrow[ld] \arrow[rd] & \\
    P \arrow[rd] & & E \times_X \hat E \arrow[ld, "p", swap] \arrow[rd, "\hat p"] & & \hat P \arrow[ld] \\
    & E \arrow[rd] & & \hat E \arrow[ld] & \\
    & & X &&
\end{tikzcd}
\]
In this case, we say that we have a T-duality triple $\bigl( (E,P), (\hat E, \hat P), u\bigr)$.
One can alternatively define a T-duality triple using a twist on the fiberwise join of $E$ and $\hat E$ that restricts to $P$ and $\hat P$ and satisfies a condition analogous to the Poincar\`e bundle condition \cite{DoveSchick:newapproach}.
The advantage of this approach is that one can more directly prove the
existence and uniqueness results for T-duals originally established in \cite{BunkeRumpfSchick}.

Coming with the T-duality relation is an isomorphism (with degree shift) between the $P$-twisted K-theory of $E$ and the $\hat P$-twisted K-theory of $\hat E$.
It is defined as the composition
\[
    K^*(E,P)
    \xrightarrow{\, p^* \,}
    K^*(E \times_X \hat E, p^*P)
    \xrightarrow{\, u^* \,}
    K^*(E \times_X \hat E, \hat p^*\hat P)
    \xrightarrow{\, \hat p_! \,}
    K^{*-n}(\hat E, \hat P).
\]
In words, we pull back along the projection $p$, apply the twist isomorphism $u$, and then push forward along $p_!$.
Twisted K-theory classifies certain physical properties of a space-time model, so such an isomorphism must exist because the physical quantities measured by the twisted K-groups on each pair must be equivalent.
More generally, there is the theory of ``T-admissible'' cohomology theories; these are those for which there is a T-duality isomorphism.
Among these are twisted K-theory and periodic twisted de Rham cohomology, the
former of which is known to be related to the classification of D-brane
charges in string theory. In the untwisted case, this was argued in
\cite{MinasianMoore} and later \cite{Witten:DbranesKtheory}. The case of
torsion twists (non-trivial torsion B-fields) was considered in
\cite{Kapustin} and the case of general twists in \cite{BM}.

Equivariant T-duality has until now not been formulated in general. 
We do so in this paper for circle bundles, extending the work of Bunke and
Schick on topological T-duality \cites{BunkeSchick05,BunkeRumpfSchick} to the
equivariant setting. 
This includes a formulation of the T-duality relation between $G$-equivariant circle bundles equipped with a $G$-equivariant twist.
We define the notion of $G$-T-admissibility and show that $G$-T-admissibility of a twisted $G$-equivariant cohomology theory implies that the T-duality transformation is an isomorphism.
The main theorem is that twisted equivariant K-theory is T-admissible, hence the T-duality transformation is an isomorphism, for all compact Lie groups:

\begin{theorem*}[Theorem \ref{thm:main}]
Let $G$ be a compact Lie group and $X$ a finite $G$-CW-complex.
If $(E,P)$ and $(\hat E, \hat P)$ are $G$-equivariant T-dual pairs of circle
bundles over $X$ then the T-duality transformation is an isomorphism: 
\[
K^*_G(E,P) \xrightarrow{\,\, \cong \,\,} K^{*-1}_G(\hat E, \hat P)
\]
\end{theorem*}

We also realize that the T-duality transformation is its own inverse
in Corollary \ref{cor:inverse_T_is_T} and that it is uniquely determined by
naturality and normalization in Corollary \ref{cor:uniqueness}.

We leave the case of equivariant principal torus bundles of fiber dimension
$n>1$ to future work. This is due to the difficulties known already in the
non-equivariant case: in this case, T-duals do not always exist, if they exist they are in
general not unique, and the Poincar\'e bundle condition for T-duality is
somewhat technical to formulate. Equivariance would add further complications
due to a zoo of possible actions and twists already of finite cyclic groups on
a single higher dimensional torus.

The interest in twisted K-theory in mathematics is largely a result of the influx of ideas from physics over the past few decades.
It was defined early in the days of K-theory, initially for torsion twists by Donovan-Karoubi \cite{DonovanKaroubi} and then later for general twists by Rosenberg \cite{Rosenberg:TKT}, but did not capture much attention until it was applied in physics.
Twisted K-theory is of interest to physicists due to its applications in
string theory, in particular as the potential home of D-brane charges in the
presence of B-fields \cites{MinasianMoore,Witten:DbranesKtheory,Kapustin,BM},
which is also why it relates to T-duality.
The significance of twisted equivariant K-theory in particular is highlighted by the results of Freed, Hopkins, and Teleman \cite{FHTtwistedktheory}, who prove a close relationship between the twisted equivariant K-theory of a compact Lie group with its conjugation action and the Verlinde algebra of its loop group. 

Defining the T-duality transformation requires an understanding of the push-forward map in twisted equivariant K-theory.
In the non-equivariant setting, the Thom isomorphism and push-forward map have been constructed in \cite{CareyWang:ThomisoTKT}, which incorporates a variety of approaches to twisted K-theory; bundle gerbe modules, bundles of Fredholm operators, and a small amount of KK-theory.
There is however little in the algebraic topology literature on the push-forward in twisted equivariant K-theory.
For this reason, we turn to non-commutative geometry, where a Thom isomorphism has been established for groupoid equivariant KK-theory \cite{Moutuou}.
We choose to define twisted equivariant K-theory in terms of the K-theory of a C*-algebra of sections encoding the twist (see Section \ref{sec:Ktheory}), and so it is straightforward to apply results about equivariant KK-theory to our twisted equivariant K-groups.
Other work on the Thom isomorphism includes the thesis of Garvey \cite{Garvey:thesis}, who proves that the Thom isomorphism in groupoid equivariant KK-theory defined in two ways, by pulling back the Bott element or by using spin representations and Clifford multiplication, are the same.

Physicists have reason to expect a generalisation of T-duality to the equivariant setting.
String theory has been formulated for sufficiently nice orbifold space-times, such as global quotient orbifolds.
Therefore one expects equivariant T-duality at the very least for finite group actions.
Early steps in this direction were the formulations of T-duality for non-free
circle actions \cites{BunkeSchicknonfree,MathaiWu}.

One of the holy grails of the equivariant theory would be the application of
topological T-duality to mirror symmetry following the formulation of
Strominger, Yau, and Zaslow \cite{SYZ}, which, however, is still a wide open
problem.

\section{Background on Topological T-Duality}

T-duality has its origins in string theory, where it is a duality between two space-time models.
Two T-dual models may a priori be very different, but are in fact physically equivalent.
This duality was first observed by studying closed strings on a cylinder of radius $R$; the cylinder is thought of as space-time where one spatial dimension has been curled up into a circle.
In this setting, the string equations of motion were found to be invariant under the transformation $R \mapsto \frac{\alpha'}{R}$, where $\alpha'$ is a constant with units of length squared.
In other words, the physics of a string on a cylinder of radius $R$ is indistinguishable from that of a cylinder of radius $1/R$.
This was generalised by Buscher to form the ``Buscher rules'', which describe how the metric and B-field change locally under T-duality \cite{Buscher87}.

In topological T-duality, we only consider the topology that underlies T-duality.
In this case, the objects of study are pairs $(E, P)$ consisting of an $S^1$-bundle $\pi \colon E \to X$ over a fixed base space $X$ together with a twist $P \to E$.
The task is always to formulate the T-duality relation and prove  existence
and uniqueness results for T-duals.
We also ask how the cohomology of the backgrounds change under T-duality; we shall see that cohomology groups that carry physical information about the space-time model should be invariant under T-duality.

The first paper on topological T-duality was by Bouwknegt, Evslin, and Mathai \cite{BEM}.
The main observation was that the presence of a non-trivial twist $P$ (the so-called \emph{H-flux}) changes the topology of space-time when taking the T-dual.
In particular, if $(E, P)$ is T-dual to $(\hat E, \hat P)$, there is an exchanging of Chern classes:
\begin{equation}\label{eqn:Texchange}
    \pi_!\bigl([P]\bigr) = c_1(\hat E),
    \,\,\,
    \hat \pi_!\bigl([\hat P]\bigr) = c_1(\hat E),
\end{equation}
where $[P] \in H^3(E;\Z)$ and $[\hat P] \in H^3(\hat E, \Z)$ are the characteristic classes classifying the twists, $c_1(-)$ denotes the first Chern class, and $\pi_!, \hat\pi_!$ are the push-forward maps for the respective $S^1$-bundles.
The main observation is that a non-trivial H-flux induces a global change in the topology of the string background.
This is notable because global topological changes are not detected by the Buscher rules, which are local in nature.

A formalised definition of topological T-duality was made by Bunke and Schick using a Thom class on an associated $S^3$-bundle, namely the sphere bundle of $L \oplus \hat L$, where $L$ and $\hat L$ are the line bundles associated with $E$ and $\hat E$ \cite{BunkeSchick05}.
In this setting, the authors confirmed the relation \eqref{eqn:Texchange} and proved that every pair $(E, P)$ has a uniquely defined T-dual, up to isomorphism.
The key part of the proof was the construction of a classifying space $R$ together with a homeomorphism $T \colon R \to R$ inducing the T-duality relation.

The same authors, with the addition of Rumpf, then formulated T-duality for principal $T^n$-bundles in \cite{BunkeRumpfSchick}, introducing the notion of T-duality triples.
They consider pairs $(E,P)$ consisting of a principal $T^n$-bundle $\pi \colon E \to X$ and a twist $P \to E$, with the additional assumption that $P$ is trivialisable when restricted to the fibers of $E$.
There was one further assumption on the characteristic class of $P$, but this turns out to be redundant \cite{DoveSchick:newapproach}.
Here, two pairs $(E,P)$ and $(\hat E, \hat P)$ are T-dual if they belong to a T-duality triple $\bigl( (E, P), (\hat E, \hat P), u \bigr)$, where $u$ is a twist morphism fitting into the following diagram:
\begin{equation}\label{eqn:Tdiag}
    \begin{tikzcd}[column sep={4em,between origins},row sep=1em]
        & p^*P \arrow[ld] 
            \arrow[rd] \arrow[rr, "u"] 
        && \hat p^* \hat P
            \arrow[ld] \arrow[rd] 
        & \\
        P 
            \arrow[rd] 
        && E \times_X \hat E
            \arrow[ld, "p", swap] \arrow[rd, "\hat p"]
        && \hat P 
            \arrow[ld] 
        \\
        & E 
            \arrow[rd] 
        && \hat E 
            \arrow[ld]
        & \\
        && X 
        &&
    \end{tikzcd}
\end{equation}
The morphism $u$ must satisfy the \emph{Poincar\'e bundle condition}, which we describe now.
First, recall that automorphisms of the trivial twist are in bijection with degree 2 cohomology of the base space. 
Now, if $x \in X$, then by choosing trivialisations of $P|_{E_x}$ and $\hat P|_{\hat E_x}$, $u$ gives an automorphism of the trivial twist on $E_x \times \hat E_x \cong T^{2n}$:
\[
    P_{\text{trivial}} 
    \cong P|_{E_x} 
    \xrightarrow{\, u \,}
    \hat P|_{\hat E_x}
    \cong
    P_{\text{trivial}}
\]
If this automorphism corresponds to the standard symplectic class of
$H^2((S^1\times S^1)^n;\Z)$ up to the choice of trivialisations of $P|_{E_x}$ and $\hat P|_{\hat E_x}$, then $u$ satisfies the Poincar\'e bundle condition.
This condition is named after the Poincar\'e line bundle, which arises in algebraic geometry and is the canonical line bundle on the product of an abelian variety with its dual.

For $n>1$, the situation is quite different to the circle case: the T-dual of a pair $(E,P)$ need not exist and if it does exist it need not be unique.
Bunke, Rumpf and Schick provide a simple criterion for when a T-dual exists and what the T-duals are \cite{BunkeRumpfSchick}*{Theorem 2.24}.
The proofs of this classification are based on the construction of a classifying space for T-duality triples; simpler proofs have since been found based on an alternative formulation of T-duality \cite{DoveSchick:newapproach}.
When the twist is not trivialisable on fibers, there is a non-classical interpretation of the T-dual as a bundle of noncommutative tori \cite{MathaiRosenberg:Tdualitytorus}.
This fits into an interpretation of T-duality via noncommutative geometry, which we won't discuss. 

Given a diagram \eqref{eqn:Tdiag} and a twisted cohomology theory $h^*(-)$, we can define the following composition:
\[
    h^*(E,P)
    \xrightarrow{\,\, p^* \,\,}
    h^*(E \times_X \hat E, p^*P)
    \xrightarrow{\,\, u^* \,\,}
    h^*(E \times_X \hat E, \hat p^*\hat P)
    \xrightarrow{\,\, \hat p_! \,\,}
    h^{*-n}(\hat E, \hat P).
\]
In words, we pull back along $p$, apply the twist automorphism $u$, and then push forward along $\hat p$.
This is called the T-duality transformation and is an essential aspect of T-duality.
Since a pair and its dual should represent equivalent space-time models, cohomological information about each model should be equivalent.
Thus the T-duality transformation should be an isomorphism for cohomological groups that carry physical information.
This is true of twisted K-theory and twisted de Rham cohomology; see the aforementioned T-duality papers.
As a general framework for this, Bunke and Schick introduced the notion of T-admissibility; the T-duality transformation can be defined for any twisted cohomology theory (satisfying some prescribed axioms) and will be an isomorphism if the theory is T-admissible \cite{BunkeSchick05}. 

In general, one would like to allow for singularities, that is, cases where the $T^n$-action is not free.
The main approach for this is to consider T-duality in the context of stacks, groupoids, and orbispaces \cites{BunkeSchicknonfree,BunkeSchickSpitzweck,Pande18}.
A full theory of T-duality for stacks, including a T-duality isomorphism for the K-theory of stacks, would of course include the equivariant case considered in this paper.
The aforementioned papers, however, do not consider this type of T-duality transformation.
\cite{BunkeSchicknonfree} considers Borel equivariant K-theory, defined by taking the non-equivariant K-theory of the Borel construction, whereas \cite{BunkeSchickSpitzweck} considers periodic twisted cohomology, which is a generalisation of de Rham cohomology.
Non-free actions have also been dealt with by passing to the Borel construction \cite{LinshawMathai}.

\section{Equivariant T-Duality: The Setup}

Let $G$ be a compact group and $X$ a locally compact $G$-space.
For the purpose of the rest of this paper, equivariant (topological) T-duality is a relationship between pairs $(E, P)$ consisting of a $G$-equivariant principal $S^1$-bundle $E \to X$ together with a $G$-equivariant twist $P \to E$.
In practice, there is a choice to be made about which model of twist to use, depending on how one wants to define twisted K-theory.
Instead of making a choice, we will define equivariant T-duality and the T-duality transformation for twists satisfying a set of prescribed axioms, which are detailed in the appendix.
When we later discuss the T-duality transformation in twisted equivariant K-theory, we shall use equivariant principal $PU(\H)$-bundles.

Consider two $G$-equivariant pairs $(E, P)$ and $(\hat E, \hat P)$ fitting into the following diagram:
\begin{equation}\label{tdualitydiagram}
\begin{tikzcd}[column sep={4em,between origins},row sep=1em]
    & 
    p^*P \arrow[ld] \arrow[rd] \arrow[rr, "u"] & & 
    \hat p^* \hat P \arrow[ld] \arrow[rd] & \\
    P \arrow[rd] & & 
    E \times_X \hat E \arrow[ld, "p", swap] \arrow[rd, "\hat p"] & & 
    \hat P \arrow[ld] \\
    & 
    E \arrow[rd, "\pi", swap] & & 
    \hat E \arrow[ld, "\hat \pi"] & \\
    & & 
    X & &
\end{tikzcd}
\end{equation}
All of the maps in the diagram are $G$-equivariant. One can therefore apply the Borel construction to the entire diagram to get a non-equivariant diagram of circle bundles and twists over the Borel construction $X \times_G EG$. This leads to our definition of equivariant T-duality:

\begin{definition}
A $G$-equivariant T-duality triple is a triple $\bigl( (E, P), (\hat E, \hat
P), u \bigr)$ fitting into a diagram of the form \eqref{tdualitydiagram} and
such that the induced non-equivariant triple over $X \times_G EG$ is a
T-duality triple (with fiber $S^1$).
\end{definition}

An obvious question is whether every T-duality triple over $X \times_G EG$ comes from an equivariant triple.
We will now prove that this is indeed the case.
Let $S(E, \hat E)$ denote the sphere bundle of $(E \times_{S^1} \C) \oplus (\hat E \times_{S^1} \C)$.
We can construct $S(E, \hat E)$ as the gluing of two mapping cylinders:
\[
    S(E,\hat E) \cong \cyl(p) \cup_{E \times_X \hat E} \cyl(\hat p).
\]
There are canonical inclusions $i \colon E \to S(E, \hat E)$ and $i \colon \hat E \to S(E, \hat E)$ that send $E$ and $\hat E$ to the ends of their respective cylinders.

\begin{proposition}\label{prop:twistiso}
There is a bijection between the set of isomorphisms $p^*P \to \hat p^*\hat P$ and the set of isomorphism classes of twists $T$ on $S(E,\hat E)$ such that $T|_E \cong P$ and $T|_{\hat E} \cong \hat P$.
\end{proposition}

\begin{proof}
Let $f \colon \cyl(p) \to E$ be the canonical map $[e, \hat e, t] 
\mapsto e$, and similarly define $\hat f \colon \cyl(\hat p) \to \hat E$. 
Let $i_E$, $i_{E \times_X \hat E}$, and $i_{\hat E}$ denote the inclusions of $E$, $E \times_X \hat E$, and $\hat E$ into $S(E, \hat E)$, respectively.

The forward direction of the bijection is defined as follows. 
Note that 
\[ 
    (f^*P)|_{E \times \hat E} = p^*P
    \quad \text{and} \quad
    (\hat f^*\hat P)|_{E \times_X \hat E} = \hat p^* \hat P.
\]
Therefore, given an isomorphism $u \colon p^*P \to \hat p^* \hat P$, we can glue together $f^*P$ and $\hat f^*\hat P$ along $E\times_X \hat E$ to obtain a twist $T = f^*P \cup_u \hat f^*P$ that appropriately restricts to $P$ and $\hat P$.

For the inverse map, given a twist $T \to S(E,\hat E)$ that restricts to $P$ and $\hat P$, we define $u \colon p^*P \to \hat p^*\hat P$ by the sequence of isomorphisms
\begin{equation}\label{eqn:twistiso}
    p^*P 
    \cong p^*(T|_E) 
    \cong T|_{E \times_X \hat E}
    \cong \hat p^*(T|{\hat E})
    \cong \hat p^*\hat P.
\end{equation}
Here, we have used that $i_E \circ p \simeq i_{E \times_X \hat E} \simeq i_{\hat E} \circ \hat p$.

Let us confirm that these constructions are inverse to each other.
Given an isomorphism $u \colon p^*P \to \hat p^*\hat P$, we first need to show that if $T = f^*P \cup_u \hat f^*\hat P$, the composition \eqref{eqn:twistiso} is equal to $u$.
In this case, $T|_{E \times_X \hat E} = p^*P \cup_u \hat p^*\hat P$ and \eqref{eqn:twistiso} becomes
\[
    p^*P \cong p^*P \cup_u \hat p^*\hat P \cong \hat p^*\hat P
\]\[
    x \mapsto x \sim u(x) \mapsto u(x),
\]
so this construction indeed returns $u$.

Now, let $T \to S(E, \hat E)$ be a twist with $T|_E \cong P$ and $T|_{\hat E} \cong \hat P$.
We show that $T \cong f^*P \cup_u \hat f^*\hat P$, where $u$ is the isomorphism \eqref{eqn:twistiso}.
Since $i_E \circ f \simeq \id$, we know that
\[
    T|_{\cyl(p)} \cong f^*(T|_E) \cong f^*P.
\]
Similarly $T|_{\cyl(\hat p)} \cong \hat f^* \hat P$.
We need these two parts to glue together according to \eqref{eqn:twistiso}, which means that the composition $(f^*P)|_{E \times_X \hat E} \cong T|_{E \times_X \hat E} \cong (\hat f^*\hat P)|_{E \times_X \hat E}$ is equal to \eqref{eqn:twistiso}.
Both isomorphisms factor through $T|_{E \times_X \hat E}$, so we can split the maps into two parts; a map $p^*P \to T|_{E \times_X \hat E}$ and another $T|_{E \times_X \hat E} \to \hat p^*\hat P$.
For the first, consider the following diagram:
\[
\begin{tikzcd}[column sep={5em,between origins}]
    p^*P 
        \arrow[rr, "P \cong T|_E"] \arrow[rd, "p = f \circ i_{E \times_X \hat E}", swap]
    && p^*(T|_E) 
        \arrow[rr, "i_E \circ p \simeq i_{E \times_X \hat E}"]
        \arrow[rd, dashed]
    && T|_{E \times_X \hat E} \\
    & (f^*P)|_{E \times_X \hat E} 
        \arrow[rr, "P \cong T|_E"]
    && (f^*T|_E)|_{E \times_X \hat E} 
        \arrow[ru, "i_E \circ f \simeq \id", swap]
    &
\end{tikzcd}
\]
The arrows are labelled with the relations that induce them.
The upper path is the first two maps in \eqref{eqn:twistiso}.
We need to show that it equals the lower path, which comes from the isomorphism $T|_{\cyl(p)} \cong f^*P$ restricted to $E \times_X \hat E$.
The dotted arrow is induced by $p = f \circ i_{E \times_X \hat E}$ and produces a commutative square.
Since
\[
    i_E \circ f \simeq \id
    \implies
    i_E \circ p = i_E \circ f \circ i_{E \times_X \hat E} \simeq i_{E \times_X \hat E},
\]
the triangle also commutes.
Therefore, the above diagram commutes. 
The same argument works for the isomorphism $T|_{E \times_X \hat E} \to \hat p^*\hat P$.
This completes the proof.
\end{proof}

This proposition can be used to formulate an equivalent definition of topological T-duality, where the twist isomorphism $p^*P \cong \hat p^* \hat P$ is replaced with a twist on $S(E,\hat E)$.
See \cite{DoveSchick:newapproach} for further details.

\begin{theorem}
The $G$-equivariant pairs $(E,P)$ and $(\hat E, \hat P)$ over $X$ (with fiber
$S^1$) are $G$-equivariantly T-dual if and only if $(E \times_G EG, P \times_G EG)$ and $(\hat E \times_G EG, \hat P \times_G EG)$ are non-equivariantly T-dual over $X \times_G EG$.
\end{theorem}

\begin{proof}
The forward direction follows by definition of equivariant T-duality.
For the reverse direction, we show that every twist morphism $u \colon p^*P \times_G EG \to \hat p^* \hat P \times_G EG$ is induced from an equivariant morphism $p^*P \to \hat p^* \hat P$.

Consider the following diagram: 
\[
\begin{tikzcd}[row sep = {1cm}]
    \left\{ \parbox{3.7cm}{\centering $G$-equivariant morphisms $p^*P \to \hat p^*\hat P$ } \right\}
        \dar["\textnormal{Borel}"] \rar["\cong"]
    & \left\{ \parbox{3.5cm}{\centering $G$-equivariant twists on $S(E, \hat E)$ that restricts to $P$ and $\hat P$} \right\} 
        \dar["\textnormal{Borel}"] \lar
    \\
     \left\{ \parbox{4.1cm}{\centering Morphisms $p^*P \times_G EG \to \hat p^*\hat P \times_G EG$ } \right\}
        \rar["\cong"] 
    & \left\{ \parbox{3.6cm}{\centering Twists on $S(E \times_G EG, \hat E \times_G EG)$ that restrict to $P \times_G EG$ and $\hat P \times_G EG$} \right\}
    \lar
\end{tikzcd}
\]
Here, the horizontal maps are a result of Proposition \ref{prop:twistiso} and the vertical maps are obtained by taking the Borel construction.
Isomorphism classes of $G$-equivariant twists on a space $X$ are in bijection with isomorphism classes of non-equivariant twists on $X \times_G EG$; they are both in bijection with $H^3_G(X;\Z)$.
We thus conclude that the vertical map on the right-hand side is an isomorphism.
So, the two horizontal maps and right-hand side vertical map are bijections, implying that the left-hand side map is a bijection.
This completes the proof.
\end{proof}

\begin{corollary}
Every T-duality triple over $X \times_G EG$ with fiber $S^1$ comes from a $G$-equivariant T-duality triple over $X$.
\end{corollary}

\begin{proof}
We need every pair over $X \times_G EG$ to be of the form $(E \times_G EG, P \times_G EG)$ for some equivariant pair $(E,G)$.
This is true because both equivariant $S^1$-bundles and bundles on $X \times_G EG$ are classified by $H^2_G(X)$ and a similar statement can be made for the twists.
\end{proof}

The theorem also implies that the existence and uniqueness properties of T-duals carry over to the equivariant setting:

\begin{corollary}
For each $G$-equivariant pair $(E,P)$ over $X$ with fiber $S^1$, there is a unique T-dual $(\hat E, \hat P)$ characterised by the relations
\[
    \pi_!\bigl([P]\bigr) = c_1(E)
    \,\, \text{and} \,\,\, 
    \hat \pi_!\bigl([\hat P]\bigr) = c_1(\hat E).
\]
\end{corollary}

\begin{example}\label{ex:verytrivial}
The most trivial $G$-equivariant T-duality triple is the triple over a trivial $G$-space $X$ consisting of trivial circle bundles and trivial twists.
\end{example}

We will later use a slightly larger class of basic T-duality triples which we
still call ``trivial'' (the duality aspects are trivial).
\begin{definition}\label{def:trivial}
  A T-duality triple of the following type is called trivial:
  \begin{enumerate}
  \item $X$ a G space with arbitrary $G$-equivariant twist $Q_0$
  \item  $E=\hat E=X\times S^1$ (with trivial
    action on the $S^1$-factor) 
  \item $Q:=\hat Q:=\pr_X^*Q_0=\pr_X^*Q_o$.
  \item    Using the natural monoidal structure, we can write
  \begin{equation*}\label{eq:tensordecomposition}
\pr_E^*Q=\pr_{\hat
    E}^*\hat Q = \pr_X^*Q_0 \otimes \pr_{S^1\times S^1}^*0
\end{equation*}
 where we write $0$ for
  the trivial twist and $\otimes$ for the monoidal product/sum on twists.
With respect to this tensor decomposition, we define the twist automorphism $u\colon \pr_E^*Q\to \pr_{\hat E}^*\hat Q$
  as
  \begin{equation*}
  u:=\id\otimes u_0 \colon  \pr_X^*Q_0 \otimes \pr_{S^1\times S^1}^*0\to  \pr_X^*Q_0 \otimes \pr_{S^1\times S^1}^*0,
\end{equation*}
where $u_0$ is 
  the standard automorphism of the trivial twist on $S^1\times S^1$ (which
  satisfies the 
  Poincar\'e bundle condition).
  \end{enumerate}
\end{definition}

\begin{example}\label{ex:noflux}
A $G$-equivariant $S^1$-bundle $E \to X$ equipped with a trivial twist is T-dual to the bundle $X \times S^1$, where $G$ acts trivially on the $S^1$-factor, equipped with a twist classified by the equivariant Chern class of $E$,
\[
    c_1(E) \in H^2_G(X) \hookrightarrow H^3_G(X \times S^1).
\]
This is summarised by saying that pairs with trivial twists are T-dual to pairs with trivial bundles.
\end{example}

\begin{example}
If $(E, P)$ and $(\hat E, \hat P)$ are T-dual, then $(E, P \otimes \pi^*Q)$ and $(\hat E, \hat P \otimes \hat\pi^*Q)$ are also T-dual, where $Q$ is an equivariant twist on $X$.
\end{example}

These are very basic examples which are needed when reducing the general
equivariant T-duality situation to basic building blocks. More specific and
more interesting examples are elaborated in Section \ref{sec:examples}, once
we have fully developed the theory.

\section{The T-Duality Transformation} \label{sec:T_trafo}

The T-duality transformation can be defined for general twisted equivariant cohomology theories.
Included in Appendix \ref{appendix} is a minimal axiomatic description of equivariant twists and twisted equivariant cohomology theories for which the T-duality transformation can be defined.
Later we will restrict our attention to twisted equivariant K-theory.

\begin{definition}\label{def:Ttransformation}
Let $h_G$ be a twisted equivariant cohomology theory and consider a
$G$-equivariant T-duality triple $\bigl( (E, P), (\hat E, \hat P), u\bigr)$
with circle fibers. The T-duality transformation is the composition $T = \hat p_! \circ u^* \circ p^*$, that is,
\[
    h^*_G(E, P) 
    \xrightarrow{\,\, p^* \,\,}
    h^*_G(E \times_X \hat E, p^*P) 
    \xrightarrow{\,\, u^* \,\,}
    h^*_G(E \times_X \hat E, \hat p^*\hat P) 
    \xrightarrow{\,\, \hat p_! \,\,}
    h^{*-1}_G(\hat E, \hat P).
\]
Note the push-forward gives a degree shift of $-1$.
\end{definition}

Given a $G$-equivariant T-duality triple $\bigl( (E, P), (\hat E, \hat P), u \bigr)$ over $X$, one can use a $G$-equivariant function $f \colon Y \to X$ to pull back the triple to a $G$-equivariant T-duality triple over $Y$.
Similarly, one can restrict along a group homomorphism $\alpha \colon H \to G$ to obtain a $H$-equivariant triple on $X$.
The T-duality transformation is natural with respect to these constructions:

\begin{proposition}
The T-duality transformation is natural with respect to group homomorphisms and continuous functions, that is, given a map $f \colon Y \to X$ and a group homomorphism $\alpha \colon H \to G$, the following diagrams commute:
\[
\begin{tikzcd}[column sep={8em,between origins},row sep=2em]
    h_G(E, P) \arrow[r, "T"] \arrow[d, "f^*", swap]
        & h_G(\hat E, \hat P) \arrow[d, "f^*"]
    \\
    h_G(f^*E, f^*P) \arrow[r, "T"]
        & h_G(f^*\hat E, f^*\hat P)
\end{tikzcd}
\quad
\begin{tikzcd}[column sep={8em,between origins},row sep=2em]
    h_G(E, P) \arrow[r, "T"] \arrow[d, "\alpha^*", swap]
        & h_G(\hat E, \hat P) \arrow[d, "\alpha^*"]
    \\
    h_H(E, P) \arrow[r, "T"]
        & h_H(\hat E, \hat P) 
\end{tikzcd}
\]
\end{proposition}

\begin{proof}
This follows directly from the naturality properties of $p^*$, $u^*$ and $\hat p_!$.
\end{proof}

\section{T-Admissibility}

Following Bunke and Schick \cite{BunkeSchick05}, we introduce a notion of T-admissibility for equivariant T-duality triples. 
Being T-admissible means that the T-duality transformation is an isomorphism for 0-cells in a $G$-CW-complex.
This will in turn imply that the T-duality transformation is an isomorphism for all finite $G$-CW-complexes.

\begin{definition}
A twisted equivariant cohomology theory is $G$-T-admissible if for each closed
subgroup $H \subseteq G$, the T-duality transformation is an isomorphism for
all pairs (with circle fiber) over the one-point space with trivial $H$-action.
\end{definition}

As a somewhat trivial example, Borel cohomology and Borel equivariant K-theory are T-admissible because they are defined via non-equivariant cohomology groups.
In general, it is difficult to prove $T$-admissibility because the equivariant T-duality over a point is still highly non-trivial; $G$-equivariant T-duality over a point is the same as T-duality over $BG$.

We will show that if a $G$-equivariant cohomology theory is T-admissible, then the T-duality transformation is an isomorphism for all $G$-CW-complexes. This will be proven by induction on the number of cells. The following two results give the base case for this induction. They essentially state when $H \subseteq G$ is a subgroup, a $H$-equivariant pair over a point can be induced up to a $G$-equivariant pair over $G/H$ and that every pair on $G/H$ arises this way.

\begin{lemma}\label{lemma:induction}
Let $H \subseteq G$ be a closed subgroup.
The function
\[
    \bigl(E_0, P_0 \bigr) 
    \longmapsto 
    \bigl(E_0 \times_H G, \Ind_H^G(P_0) \bigr)
\]
induces a bijection between isomorphism classes of $H$-equivariant pairs over
a $H$-space $X$ and the isomorphism classes of $G$-equivariant pairs over $X \times_H G$.
\end{lemma}

\begin{proof}
Let $(E,P)$ be a $G$-equivariant pair over $X \times_H G$.
If $E_0 \to X$ is a $H$-equivariant principal $S^1$-bundle then $E_0 \times_H G \to X \times_H G$ is a $G$-equivariant principle $S^1$-bundle, and this construction induces a bijection between isomorphism classes of $H$-equivariant bundles on $X$ and $G$-equivariant bundles on $X \times_H G$.
Therefore, we can assume that $E \cong E_0 \times_H G$ for a unique (up to isomorphism) $H$-equivariant $S^1$-bundle $E_0 \to X$.
By the twist axioms, there is an induction construction
\[
    \Ind_H^G \colon \twist_H(E_0) \to \twist_G(E_0 \times_H G)
\]
that induces a bijection on isomorphism classes. 
Therefore, $P \in \twist_G(E)$ is isomorphic to $\Ind_H^G(P_0)$ for a unique (up to isomorphism) $P_0 \in \twist_H(E_0)$.
We conclude that $(E,P)$ is isomorphic to $(E_0 \times_H G, \Ind_H^G(P_0))$.
\end{proof}

This lemma tells us that we can induce $H$-equivariant pairs over a $H$-space $X$ to $G$-equivariant pairs over $X \times_H G$:
\[
P \to E \to X 
\quad \leadsto  \quad 
\Ind_H^G(P) \to E \times_H G \to X \times_H G.
\]
In the special case where $X$ is a point, Lemma \ref{lemma:induction} says that this is an equivalence between $H$-equivariant pairs over a point and $G$-equivariant pairs over $G/H$.

\begin{theorem}
If an equivariant cohomology theory $h^*_-(-)$ is $G$-T-admissible and $H
\subseteq G$ is a closed subgroup, then the T-duality transformation is an
isomorphism for all pairs over the base space $G/H$.
\end{theorem}

\begin{proof}
Let $(E, P)$ and $(\hat E, \hat P)$ be $G$-equivariant T-dual pairs over $G/H$. Let $(E_0, P_0)$ and $(\hat E_0, \hat P_0)$ be the corresponding $H$-equivariant pairs over a point.
Consider the following diagram:
\[
\begin{tikzcd}
    h^*_G(E,P) \arrow[r, "T"] \arrow[d, "\cong"]
    & h^{*-1}_G(\hat E, \hat P) \arrow[d, "\cong"] \\
    h^*_H(E_0, P_0) \arrow[r, "T"] \arrow[r, "\cong", swap] 
    & h^{*-1}_H(\hat E_0, \hat P_0).
\end{tikzcd}
\]
The diagram commutes because the T-duality transformation commutes with the
induction isomorphism.
The bottom arrow is an isomorphism by T-admissibility.
The vertical maps are isomorphisms by the induction axiom.
Hence the top arrow is an isomorphism, as required.
\end{proof}

\begin{theorem}\label{thm:Tadmissibleiso}
If a twisted $G$-equivariant cohomology theory is T-admissible then the T-duality transformation is an isomorphism for finite $G$-CW-complexes.
\end{theorem}
\begin{proof}
This is proven in the same way as in \cite{BunkeSchick05}, except now we use induction on the number of $G$-CW-cells.
The base case is true by the previous lemma.
Then one checks that the T-duality transformation is natural with respect to pullbacks and the boundary operator in the Mayer-Vietoris sequence.
The induction step is proven by attaching a cell and using the 5-lemma on the resulting Mayer-Vietoris sequence.
\end{proof}

\begin{example}
$\Z$-equivariant K-theory offers a baby example of T-admissibility, since, for cohomological reasons, there are no non-trivial $\Z$-equivariant pairs over a point.
\end{example}

\section{Twisted Equivariant K-Theory}\label{sec:Ktheory}

For the remainder of the paper, we will focus our attention on twisted equivariant K-theory.
In this section, we provide a definition and describe how one defines the push-forward along principal $S^1$-bundles.
Such maps are required to define the T-duality transformation.

Let $X$ be a space acted on by a compact group $G$ and let $P \to X$ be a stable $G$-equivariant principal $PU(\H)$-bundle.
For simplicity, we call these bundles $G$-equivariant twists.
Stable equivariant projective unitary bundles are defined, for instance, in \cite{BEJU:universaltwist}*{Def 2.2}. 
We provide an in-depth discussion of these in Appendix \ref{appendix}.

Let $\mathcal K$ denote the space of compact operators on $\H$.
$PU(\H)$ acts on $\mathcal K$ via conjugation, so that for each
$G$-equivariant $PU(\H)$-bundle $P \to X$ there is an associated
$G$-equivariant bundle of $C^*$-algebras (isomorphic to $\mathcal K$) $P \times_{PU(\H)} \mathcal K$.
Let $\Gamma_0(P \times_{PU(\H)} \mathcal K)$ denote the C*-algebra of sections of $P \times_{PU(\H)} \mathcal K$ that vanish at infinity. This is a $G$-equivariant C*-algebra with the action given by
\[
    (g \cdot \sigma)(x) = \sigma(x\cdot g)g^{-1},
\]
where $g \in G$, $\sigma \in \Gamma_0(P \times_{PU(\H)} \mathcal K)$ and $x \in X$.

\begin{definition}\label{def:equiv_twisted_K}
The $P$-twisted $G$-equivariant K-theory of $X$ is defined as the $G$-equivariant K-theory of $\Gamma_0(P \times_{PU(\H)} \mathcal K)$:
\[
    K^*_G(X,P) := K^G_* \bigl( \Gamma_0(P \times_{PU(\H)} \mathcal{K}) \bigr).
\]
\end{definition}

This is motivated by Rosenberg's definition of twisted K-theory \cite{Rosenberg:TKT}*{\textsection 2}.
The equivariant version appears in, for instance, \cite{Karoubi:oldandnew}*{\textsection 5.4} and \cite{Meinrenken}*{\textsection 2.3}. 
There are of course other formulations of twisted equivariant K-theory; for example, via equivariant sections of the bundle of Fredholm operators associated with $P$ \cite{AtiyahSegal:TKT}*{\textsection 7}.
The formulation we have chosen allows us to access tools in non-commutative geometry.
Indeed, by the Green-Julg theorem, because $G$ is compact, twisted equivariant K-theory simply becomes the ordinary K-theory of some $C^*$-algebra, namely the crossed product of $G$ with the above algebra of sections.

To define a push-forward map in any cohomology theory, one needs a Thom isomorphism.
We shall only briefly explain why a Thom isomorphism exists in twisted
equivariant K-theory; further details are available in the first author's PhD thesis \cite{Dove:phd}. 
The idea is that the Thom isomorphism already exists for groupoid equivariant KK-theory \cite{Moutuou} and that our twisted K-groups can be interpreted as $(G \rtimes X)$-equivariant KK-groups:
\[
    KK^G_* \bigl( \C,  \Gamma_0(P_{\mathcal K}) \bigr)
    \cong KK^{G \rtimes X}_*\bigl( C_0(X), \Gamma_0(P_{\mathcal K}) \bigr)
\]
Here, we are writing $P_{\mathcal K} := P \times_{PU(\H)} \mathcal K$.
We refer the reader to \cite{LeGall} for the definition of groupoid equivariant K-theory.
Now we can apply the following theorem, which is just a specific case of \cite{Moutuou}*{Theorem 8.1}:

\begin{theorem}
Let $\pi \colon V \to X$ be a K-oriented real $G$-equivariant vector bundle of rank $n$.
There exists a natural Thom isomorphism
\[
   K_G^*(X,P) \cong K_G^{*+n}(V, \pi^*P).
\]
\end{theorem}

We remark that a Thom isomorphism exists even if the vector bundle is not K-oriented; one just needs to introduce an additional twist on $X$ coming from the Clifford bundle of $V$.
This isn't required for our purposes.

With the Thom isomorphism established, it is straightforward to define the push-forward along K-oriented maps $f \colon X \to Y$.
One uses the standard Pontryagin-Thom construction: choose an embedding of $X$ into $\R^N$, then factor $f$ through $\R^N \times Y$ and use the Thom isomorphism for the normal bundle together with the Pontryagin-Thom collapse map.
Let us carry out this process for $G$-equivariant principal $S^1$-bundles.
These are all K-oriented and the push-forward can be constructed relatively explicitly. 

Let $p \colon E \to X$ be a $G$-equivariant principal $S^1$-bundle.
We can always factor $p$ through the associated complex line bundle:
\begin{center}
\begin{tikzpicture}
    \node (E) at (0,0) {$E$};
    \node (L) at (1.5,1) {$E \times_{S^1} \C$};
    \node (X) at (3,0) {$X$};
    \draw [->] (E) -- (X) node [midway, above] {\scriptsize $p$};
    \draw [right hook ->] (E) -> (L) node [midway, above, ,xshift=-1mm] {\scriptsize $i$};
    \draw [->] (L) -> (X) node [midway, above, xshift=1mm] {\scriptsize $\pi$};
\end{tikzpicture}
\end{center}
The push-forward $p_!$ is defined by pushing forward along the embedding $i$ and then using the Thom isomorphism for $\pi$.
The second part is clear: the Thom isomorphism in equivariant K-theory exists for locally compact $G$-spaces \cite{Segal:EKT}*{\textsection 3}.
Thus we focus on the push-forward along the embedding $i \colon E \to E \times_{S^1} \C$. 

\begin{lemma}
Let $p \colon E \to X$ be a $G$-equivariant principal $S^1$-bundle.
The normal bundle of the embedding $i\colon E \to E \times_{S^1} \C$ is
isomorphic to $E \times \R$ with trivial action of $G$.
\end{lemma}

\begin{proof}

For notational ease, let $\tilde E = E \times_{S^1} \C$.
We write elements of $T\tilde E$ as pairs $(e, u)$ with $e \in \tilde E$ and $u \in T_e\tilde E$.
Choose the following inner product on $\tilde E$:
\[
    \bigl\langle [e, z_1], [e, z_2] \bigr\rangle := \operatorname{Re}\bigl(\bar{z}_1z_2\bigr).
\]
This comes from the standard inner product on $\R^2$ after the canonical identification $\C = \R^2$.
This inner product satisfies the following properties
\begin{gather}
    i(E) = \bigl\{ e \in E \times_{S^1} \C : \lVert e \rVert = 1 \bigr\} \label{prop1} \\
    i_*(T_v E) = \bigl\{ (e,u) \in (T_v\tilde E)|_{i(E)} : \langle e, u \rangle= 0 \bigr\} \label{prop2}
\end{gather}
We explain property \eqref{prop2}: $T_vE$ and $T_v\tilde E$ denote the vertical sub-bundle of the tangent bundles on $E$ and $\tilde E$, respectively.
The product $\langle e,u \rangle$ makes sense after identifying $(T_v\tilde E)_e$ with $\tilde E_{p(e)}$.
This is possible because there is a canonical isomorphism $T_v\tilde E \cong \pi^*\tilde E$.
Property \eqref{prop1} is straightforward to show.
For \eqref{prop2}, it is a matter of checking the condition locally.

These conditions are motivated by the standard embedding $S^1 \hookrightarrow \C$.
The image of $S^1$ is precisely the unit complex numbers and the tangent vectors to $u \in S^1$ are precisely those orthogonal to $u$. 

From here, we can write the map directly.
Let $P \colon T\tilde E \to T_v\tilde E$ be a choice of projection map onto the vertical sub-bundle.
Consider the map
\begin{equation}\label{normalbundlemap}
    T\tilde E|_{i(E)} \to E \times \R, 
    \quad
    (e, u) \mapsto \bigl(e, \langle e, Pu \rangle \bigr),
\end{equation}
again making the identification $(T_v\tilde E)_e = \tilde E_{\pi(e)}$ for the inner product to make sense.
An element $(e,u) \in i_*(TE)$ is mapped to $(e, 0)$ because of \eqref{prop2}, and thus \eqref{normalbundlemap} descends to a map on the normal bundle $T\tilde E|_{i(E)} / i_*(TE)$.
This is an isomorphism; the inverse sends $(e, t)$ to $(e, te)$, where $te \in E_{\pi(e)} \cong (T_v\tilde E)_e$, hence lives in the correct space, $T_e\tilde E$.
\end{proof}

With this done, we define the push-forward along $i \colon E \to E \times_{S^1} \C$.
Using the tubular neighbourhood theorem, choose a neighbourhood $U \supseteq i(E)$ such that $U \cong N_i$, where $N_i \cong E \times \R$ is the normal bundle.
We can in fact choose $U = E \times_{S^1} \C^*$, because $E \times \R \cong E \times_{S^1} \C^*$ via a map $(e, t) \mapsto [e, \gamma(t)]$, where $\gamma \colon \R \to (0,\infty)$ is any homeomorphism.
Then, the push-forward is the composition
\[
    i_! \colon K_G^*(E) \cong K_G^{*+1}(E \times \R) \cong K_G^{*+1}(E \times_{S^1} \C^*) \to K_G^{*+1}(E \times_{S^1} \C).
\]
The final map is the extension map; it can be described by extending compactly supported sections on $E \times_{S^1} \C^*$ to $E \times_{E^1} \C$ by mapping $[e,0]$ to $0$.

Therefore, the final push-forward map is the composition
\[
    p_! \colon K^*_G(E) \xrightarrow{\,\, i_! \,\,} K_G^{*+1}(E \times_{S^1} \C) \xrightarrow{\,\, \pi_! \,\,} K_G^{*+1}(X),
\]
where $\pi_!$ is the inverse of the Thom isomorphism $K_G(X) \cong K_G(E \times_{S^1} \C)$.
An important remark is that this works just the same for twisted equivariant K-theory; no further orientability conditions are required to define a Thom isomorphism in twisted equivariant K-theory and since this push-forward is defined using two Thom isomorphisms, everything works in this setting as well.

\begin{lemma}\label{lem:trivialtrans}
In K-theory, the T-duality transformation for the trivial T-duality diagrams
described in Definition \ref{def:trivial} can be described explicitly:

First note that we can rewrite the suspension isomorphism in Corollary \ref{cor:suspension} for the product with
$S^1$ (with trivial $G$-action) as a K\"unneth isomorphism
\begin{equation*}
  K^*_G(X\times S^1, \pr_X^*Q) \xrightarrow{\cong} K^*_G(X,Q)\otimes K^*(S^1),
\end{equation*}
where we use the $\mathbb{Z}/2$-graded tensor product, and consequently also
\begin{equation*}
  K^*_G(X\times S^1\times S^1, \pr_X^*Q) \xrightarrow{\cong} K^*_G(X,Q)\otimes
  K^*(S^1\times S^1).
\end{equation*}
The T-duality transformation for the twist automorphism pulled back from
$S^1\times S^1$ then becomes
\begin{equation}\label{eq:trivial_T}
\begin{tikzcd}
    K^*_G(X \times S^1,\pr_X^*Q) 
        \rar["\cong"] \dar["T",swap]
    &K^*_G(X ,Q) \otimes K^*(S^1)
        \dar["\id \otimes T"] \\
    K^{*-1}_G(X \times S^1,\pr_X^*Q)
        \rar["\cong"]
    &K^{*}_G(X, Q) \otimes K^{*-1}(S^1).
\end{tikzcd}
\end{equation}

In particular, this is an isomorphism because the T-duality transformation is an isomorphism for non-equivariant T-duality triples.
\end{lemma}
\begin{proof}
  To identify the ``suspension isomorphism'' recall that we define in
  Definition \ref{def:equiv_twisted_K} the twisted equivariant K-theory as the
  equivariant K-theory of the $C^*$-algebra of sections of
  $\pr_X^*Q\times_{PU(\H)}\mathcal{K}$. Now, because the bundle is pulled back from $X$
  to $X\times S^1$, for this section algebra we have the canonical isomorphism
  \begin{equation*}
    \Gamma_0(\pr_X^*Q\times_{PU(\H)}\mathcal{K}) \cong \Gamma_0(Q\times_{PU(\H)}\mathcal{K})
    \otimes C(S^1).
  \end{equation*}
  Using stability, we can also rewrite as an equally natural isomorphism
  \begin{equation*}
    \Gamma_0(\pr_X^*Q\times_{PU(\H)}\mathcal{K}) \cong \Gamma_0(Q\times_{PU(\H)}\mathcal{K})
    \otimes C(S^1,\mathcal{K}).
  \end{equation*}

  For the computation of the T-duality transformation we observe that, by naturality, under
  this isomorphism we get the following commutative diagrams for the pullback
  and push-forward maps:
\[
\begin{tikzcd}
    K^*_G(X \times S^1,\pr_X^*Q) 
        \rar["\cong"] \dar["(\id_X\times \pr_1)^*",swap]
    &K^*_G(X ,Q) \otimes K^*(S^1)
        \dar["\id \otimes \pr_1^*"] \\
    K^{*-1}_G(X \times S^1\times S^1,\pr_X^*Q)
        \rar["\cong"]
    &K^{*}_G(X, Q) \otimes K^{*-1}(S^1\times S^1).
\end{tikzcd}
\]

\[
\begin{tikzcd}
    K^*_G(X \times S^1\times S^1,\pr_X^*Q) 
        \rar["\cong"] \dar["(\id_X\times\pr_2)_!",swap]
    &K^*_G(X ,Q) \otimes K^*(S^1\times S^1)
        \dar["\id \otimes (\pr_2)_!"] \\
    K^{*-1}_G(X \times S^1,\pr_X^*Q)
        \rar["\cong"]
    &K^{*}_G(X, Q) \otimes K^{*-1}(S^1).
\end{tikzcd}
\]

To finish the proof of the commutativity of \eqref{eq:trivial_T} we have to
analyse the automorphism of twists and what it induces on K-theory. In our
model, this is given by an automorphism of the bundle of operator algebras
$\pr_X^*Q\times_{PU(\H)}\mathcal{K} \otimes \mathcal{K}$ over $X\times
S^1\times S^1$ which is the tensor product of the identity on the first tensor
factor and the ``standard'' automorphism $u$ of $S^1\times S^1\times
\mathcal{K}$ which satisfies the Poincar\'e bundle condition. Consequently,
the induced map on the section algebras and then on K-theory splits as the
tensor product of $\id$ on the $X$-factor and the standard automorphism of the
trivial twist on the $S^1\times S^1$-factor. The result follows.

\end{proof}

\section{\texorpdfstring{$\Z_n$}{Zn}-Equivariant K-Theory}\label{sec:cyclic}\label{sec:ZnTadmiss}

The simplest groups to consider are the finite cyclic groups, yet even in this case it is non-trivial to show that the T-duality transformation is an isomorphism. In this section, we prove Theorem \ref{thm:ZnTadmiss}, which states that $\Z_n$-equivariant K-theory is T-admissible.
We must start by investigating the possible $\Z_n$-equivariant T-duality triples over a point and the relevant K-theory groups.

Since $H^2_{\Z_n}(*;\Z) \cong \Z_n$, there are $n$ isomorphism classes of $\Z_n$-equivariant principal $S^1$-bundle over a point.
For each $k \in \Z_n$, let $E_k \to *$ denote the corresponding bundle.
Explicitly, $E_k$ is isomorphic to $S^1$ with the action $\xi \cdot z = \xi^kz$, where $\xi \in \Z_n \subseteq S^1$ is the generator.
The Gysin sequence for $E_k$ gives
\[
    \dotsm
    \to H^3(B\Z_n;\Z)
    \to H^3_{\Z_n}(E_k;\Z)
    \to H^2(B\Z_n;\Z)
    \xrightarrow{\cdot k} H^4(B\Z_n;\Z) 
    \to \dotsm,
\]
which, because the odd cohomology groups of $\Z_n$ are trivial, allows us to conclude that 
\[
    H^3_{\Z_n}(E_k;\Z)
    \cong \{ l \in \Z_n \mid kl = 0 \} \subseteq \Z_n.
\]
Thus, every $\Z_n$-equivariant T-duality pair over a point is of the form $(E_k, l)$ with $kl \equiv 0 \mod n$.
The pair $(E_k, l)$ is T-dual to $(E_l, k)$, so our task is to show that the T-duality transformation gives an isomorphism $K^*_{\Z_n}(E_k,l) \cong K^{*-1}_{\Z_n}(E_l, k)$.
First, let us calculate these K-theory groups.

\begin{lemma}\label{lem:EkKtheory}
Let $E_k = S^1$ with the $\Z_n$-action defined by $\Z_n \xrightarrow{ \times k } \Z_n \subseteq S^1$.
Let $\tau_\ell$ be the $\Z_n$-equivariant twist on $E_k$ classified by $\ell \in \ker(\Z_n \xrightarrow{\times k} \Z_n) \cong H^3_{\Z_n}(E_k; \Z)$. Then,
\[
    K^0_{\Z_n}\bigl(E_k, \tau_\ell \bigr) 
        \cong R\bigl(\Z_{\gcd(n,k)}\bigr)^{\xi^{\frac{\gcd(n,k) \ell}{n}}} \text{and}
\]
\[ 
    K^1_{\Z_n}\bigl(E_k, \tau_\ell\bigl) 
        \cong R\bigl(\Z_{\gcd(n,k)}\bigr) / \langle 1 -\xi^{\frac{\gcd(n,k) \ell}{n}} \rangle,
\]
where $\xi$ is the representation generating $R\bigl(\Z_{\gcd(n,k)}\bigr)$.
\end{lemma}

We remark on why $\gcd(n,k)\ell/n$ is an integer. Indeed, the kernel of $\Z_n \xrightarrow{\times k} \Z_n$ is the subgroup generated by $n / \gcd(n,k)$ and so since $\ell$ is an element of this subgroup, it must be a multiple of $n / \gcd(n,k)$.
This implies that $\gcd(n,k)\ell/n$ is an integer.

\begin{proof}
The groups are computed using a Mayer-Vietoris argument; the same technique is used in \cite{FHTtwistedktheory}*{\textsection 1}.
For convenience, write $d = \gcd(n,k)$.
$E_k$ has a $\Z_n$-CW-structure consisting of a 0-cell $e^0 \times \Z_n / \Z_d$ and a 1-cell $e^1 \times \Z_n / \Z_d$.
Now and throughout the proof, we implicitly identify $\Z_d$ with the subgroup of $\Z_n$ generated by $n/d$.
\begin{center}
\begin{figure}
\begin{tikzpicture}
    
    \def \ptsize {0.1};
    \def \rad {2} 
    \def \blen {32}; 
    \def \slen {23}; 

    \draw (0,0) circle [radius=2];

    \foreach \n in {0,...,3}
    {
        \draw[ultra thick, myblue]
            (70+90*\n-\slen:\rad+0.1) arc 
            (70+90*\n-\slen:70+90*\n+\slen:\rad+0.1);

        \draw[ultra thick, mygreen] 
            (115+90*\n-\blen:\rad-0.1) arc 
            (115+90*\n-\blen:115+90*\n+\blen:\rad-0.1);

        \draw[ultra thick, mypink]
            (70+90*\n-\slen:\rad) arc
            (70+90*\n-\slen:70+90*\n-\slen+10:\rad);

        \draw[ultra thick, mypink]
            (70+90*\n+\slen:\rad) arc
            (70+90*\n+\slen:70+90*\n+\slen-10:\rad);

        \draw [fill] (70+90*\n:\rad) circle [radius=\ptsize];

        \node at (70+90*\n+\slen-5:2.5) {$\xi^{\frac{d\ell}{n}}$};
        \node at (70+90*\n-\slen+5:2.4) {$1$};
    };

    \node[myblue, right] at (3,1) {\large $\mathbf{U \simeq \Z_n/\Z_d}$};
    \node[mygreen, right] at (3,0) {\large $\mathbf{V \simeq \Z_n/\Z_d}$};
    \node[mypink, right] at (3,-1) {\large $\mathbf{U\cap V \simeq \Z_n/\Z_d \cup \Z_n/\Z_d}$};
\end{tikzpicture}
    \caption{\label{fig:coverEk}The open cover $\{U,V\}$ used to calculate $K^*_{\Z_n}(E_k, \tau_\ell)$.}
\end{figure}
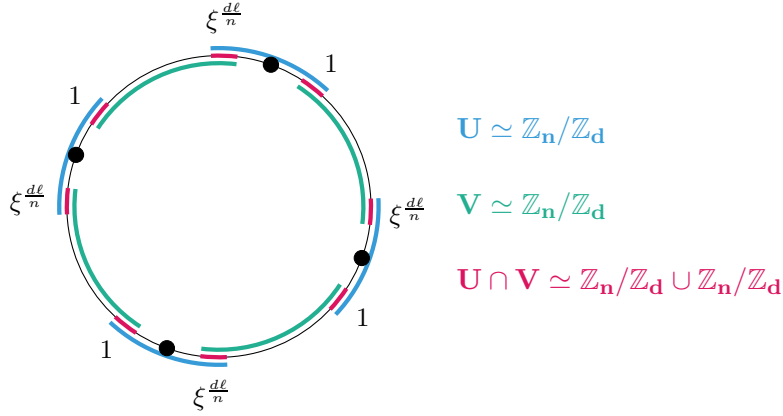
\end{center}

Let $U$ be a small open neighbourhood around the 0-cell and $V$ an open set containing $E_k \setminus U$ disjoint from the 0-cell; see Figure \ref{fig:coverEk}.
Then $U \simeq \Z_n/\Z_d \simeq V$ and $U\cap V$ is a disjoint union of two copies of $\Z_n / \Z_d$.
A twist on $E_k$ can be modelled as a $\Z_n$-equivariant line bundle on the intersection; this is equivalent to two choices of $\Z_d$-representation.
Up to stable isomorphism (of Hitchin gerbes), we can choose one of these representations to be trivial.
To represent $\tau_\ell$, the remaining $\Z_d$-representation is chosen to be $\xi^{d\ell/n}$.
The twist is depicted in Figure \ref{fig:coverEk}. 

Noting that $K^0_{\Z_n}(\Z_n/\Z_d) = R(\Z_d)$ and $K^1_{\Z_n}(\Z_n/\Z_d)= 0$, the Mayer-Vietoris sequence for $E_k= U \cup V$ is
\[
0 
\rightarrow
K^0_{\Z_n}(E_k, \tau_\ell)
\rightarrow
R(\Z_d)^2 
\xrightarrow{\,\, \left(\begin{smallmatrix} 1 & -\xi^{d\ell/n} \\ 1 & -1 \end{smallmatrix}\right) \,\, }
R(\Z_d)^2 
\rightarrow
K^1_{\Z_n}(E_, \tau_\ell)
\rightarrow
0.
\]
The map in the center is a result of choosing trivialisations of $\tau_\ell$ over $U$ and $V$; more details are found in \cite{FHTtwistedktheory}*{\textsection 1}.
From the sequence, we conclude that
\[
K^0_{\Z_n}(E_k, \tau_\ell) 
\cong
\ker\left(\begin{smallmatrix} 1 & -\xi^{d\ell/n} \\ 1 & -1 \end{smallmatrix}\right)
\cong
R(\Z_d)^{\xi^{d\ell/n}}
\text{ and}
\]
\[
K^1_{\Z_n}(E_k, \tau_\ell)
\cong
\coker\left(\begin{smallmatrix} 1 & -\xi^{d\ell/n} \\ 1 & -1 \end{smallmatrix}\right)
\cong
R(\Z_d) / \langle 1-\xi^{d\ell/n}\rangle.
\]
The final isomorphism is given by $[(p,q)] \mapsto [p-q]$.
\end{proof}

It will also be useful to know what the restriction maps are:

\begin{lemma}\label{lem:Ekrest}
Consider again the assumptions made in the previous lemma and let $m$ be an integer dividing $n$.
Restricting the K-theory groups of $E_k$ along the inclusion $\Z_m \hookrightarrow \Z_n$ induce the following diagrams:
\begin{equation}\label{diag:restK0}
\begin{tikzcd}
    K^0_{\Z_n}(E_k, \tau_\ell) \arrow[r, "\cong"] \arrow[d] 
        &  R\bigl(\Z_{\gcd(n,k)}\bigr)^{\xi^{\frac{\gcd(n,k) \ell}{n}}} \arrow[d, shift right = 5] \\
    K^0_{\Z_m}(E_k, \tau_\ell) \arrow[r, "\cong"] 
        & R\bigl(\Z_{\gcd(m,k)}\bigr)^{\eta^{\frac{\gcd(m,k) \ell}{m}}}
\end{tikzcd}
\end{equation}
\begin{equation}\label{diag:restK1}
\begin{tikzcd}
    K^1_{\Z_n}(E_k, \tau_\ell) \arrow[r, "\cong"] \arrow[d] 
        & R\bigl(\Z_{\gcd(n,k)}\bigr) / \langle 1 -\xi^{\frac{\gcd(n,k) \ell}{n}} \rangle \arrow[d] \\
    K^1_{\Z_m}(E_k, \tau_\ell) \arrow[r, "\cong"] 
        & R\bigl(\Z_{\gcd(m,k)}\bigr) / \langle 1 -\eta^{\frac{\gcd(m,k) \ell}{m}} \rangle 
\end{tikzcd}
\end{equation}
The map on the right-hand side of \eqref{diag:restK0} is a restriction of $R(\Z_{\gcd(n,k)}) \to R(\Z_{\gcd(m,k)})$, which is in turn induced by the inclusion $\Z_{\gcd(m,k)} \hookrightarrow \Z_{\gcd(n,k)}$.
The map sends $\xi$ to $\eta = \xi^{\frac{\gcd(n,k)}{\gcd(m,k)}}$.
The vertical map on the right-hand side of \eqref{diag:restK1} is given by 
\[
    [p(\xi)] \longmapsto 
    \bigl[ (1 + a + a^2 + \dotsm + a^{\frac{n \gcd(m,k)}{m \gcd(n,k)} - 1} ) p(\eta) \bigr],
\]
where $a = \eta^{\frac{\gcd(n,k)\ell}{n}}$.
Here, we write the elements of the representation ring as polynomials, with $p(x)$ denoting a polynomial in $x$.
\end{lemma}

\begin{proof}

We continue using the notation established in the previous lemma's proof.
In addition to writing $d = \gcd(n,k)$, we write $d' = \gcd(m,k)$. Let $\xi$ be the generator of $R(\Z_d)$ and $\eta = \xi^{d/d'}$ the generator of $R(\Z_{d'})$.
Elements of these rings will be written as polynomials $p(\xi)$ or $p(\eta)$.

Start by observing that $K_{\Z_m}(\Z_n / \Z_d) \cong R(\Z_{d'})^{\frac{nd'}{md}}$.
This is a result of the orbit-stabiliser theorem: we have $\Z_m$ acting on $n/d$ points with stabiliser $\Z_m \cap \Z_d = \Z_{d'}$ (identifying these groups with subgroups of $\Z_n$).
Consider the Mayer-Vietoris sequence for the $\Z_m$-equivariant K-theory alongside the sequence considered in the previous proof:
\[
\begin{tikzcd}[column sep = 15]
    0 \arrow[r] 
        & K^0_{\Z_n}(E_k, \tau_\ell)
        \arrow[r] \arrow[d]
        & R(\Z_d)^2 
        \arrow[r, "\sbm{1 \amp -\xi^{d\ell/n} \\ 1 \amp -1}"] \arrow[d]
        & R(\Z_d)^2 
        \arrow[r] \arrow[d]
        & K^1_{\Z_n}(E_k, \tau_\ell) 
        \arrow[r] \arrow[d]
        & 0 
    \\
    0 \arrow[r]
        & K^0_{\Z_m}(E_k, \tau_\ell)
        \arrow[r]
        & \bigl[R(\Z_{d'})^{\frac{nd'}{md}}\bigr]^2
        \arrow[r, "\Phi"]
        & \bigl[R(\Z_{d'})^{\frac{nd'}{md}}\bigr]^2
        \arrow[r]
        & K^1_{\Z_m}(E_k, \tau_\ell) 
        \arrow[r]
        & 0 
\end{tikzcd}
\]
The two vertical arrows in the center are given by 
\[
\bigl( p(\xi), q(\xi) \bigr) \mapsto \bigl( p(\eta), \dotsc, p(\eta), q(\eta), \dotsc, q(\eta) \bigr).
\]
The map $\Phi$ can be defined as
\begin{align*}
   &\Phi \bigl(p_1, \dots, p_j, q_1, \dots, q_j \bigr)  \\
        \hspace{5mm}&= \bigl( p_1 - a q_1, p_2 - a q_2, \dotsc, p_j - a q_j, p_1 - q_2, p_2-q_3, \dotsc, p_{j-1} - q_j, p_j - q_1 \bigr),
\end{align*}
where, for brevity, $j = \frac{nd'}{md}$ and $a = \eta^{d\ell/n}$.
Applying  Lemma \ref{lem:EkKtheory}, we get
\[
K^0_{\Z_m}(E_k, \tau_\ell) = R(\Z_{d'})^{\eta^{d'\ell/m}} 
\quad \text{and} \quad
K^1_{\Z_m}(E_k, \tau_\ell) = \frac{R(\Z_{d'})}{\langle 1 -\eta^{d'\ell/m} \rangle}.
\]
We can identify these with $\ker \Phi$ and $\coker\Phi$ as follows:
\[
    \ker \Phi 
        \xrightarrow{\,\, \cong \,\,} 
        R(\Z_{d'})^{\eta^{d'\ell/m}}, 
    \quad
    (p_1, \dotsc, p_j, q_1, \dotsc, q_j) 
        \longmapsto p_1,
\]
\begin{align*}
    & \hspace{20mm}
    \coker \Phi 
    \xrightarrow{\,\, \cong \,\,} 
    \frac{R(\Z_{d'})}{\langle 1 -\eta^{d'\ell/m} \rangle}\tag{\theequation}\label{cokermap} \\
    & 
    [(x_1, \dotsc, x_j, y_1, \dotsc, y_j)] \\
    & 
    \hspace{10mm}\longmapsto [(x_1 + a x_j + a^2x_{j-1} \dotsm + a^{j-1}x_2) - (a y_j + a^2y_{j-1} + \dotsm + a^j y_1)].
\end{align*}
To make sense of this: first note that if $\Phi(p_1, \dotsc, p_j, q_1, \dotsc, q_j) = 0$, then
\[
    p_1 = aq_1 = ap_j = a^2q_j = \dotsm = a^j p_1 
    \text{ and}
\]
\[
    a^j = \bigl(\eta^{d\ell/n}\bigr)^{\frac{nd'}{md}} = \eta^{d'\ell/m},
\]
so $p_1 \in R(\Z_{d'})^{\eta^{d'\ell/m}}$.
Furthermore, \eqref{cokermap} is well defined, since
\begin{align*}
   &\Phi(p_1, \dots, p_j, q_1, \dots, q_j)  \\
        &\hspace{5mm}
        = (p_1 - a q_1, p_2 - a q_2, \dotsc, p_j - a q_j, p_1 - q_2, p_2-q_3, \dotsc, p_{j-1} - q_j, p_j - q_1) \\
        &\hspace{20mm}
        \longmapsto (p_1 - a q_1) + a(p_j - a q_j) + \dotsm a^{j-1}(p_2-a q_2) \\
        &\hspace{40mm} 
        -[ a(p_j-q_1) + a^2(p_{j-1}-q_j) + \dotsm + a^j(p_1-q_2) ] \\
        &\hspace{25mm} 
        = (1-a^j) p_1  \\
        &\hspace{25mm} 
        =  (1 - \eta^{d'\ell/m}) p_1.
\end{align*}
It is straightforward to confirm that the inverse map is $[p] \mapsto [(p,0,\dotsc, 0)]$. 

Now, to determine the map on $K^0$-groups, consider the following:
\[
\begin{tikzcd}
    K^0_{\Z_n}(E_k, \tau_\ell) \arrow[r, "\cong"] \arrow[d] 
        & \ker\left(\begin{smallmatrix} 1 & -\xi^{d\ell/n} \\ 1 & -1 \end{smallmatrix}\right) \arrow[r, "\cong"] \arrow[d]
        &  R\bigl(\Z_d\bigr)^{\xi^{\frac{d \ell}{n}}} \arrow[d, shift right = 5] \\
    K^0_{\Z_m, \tau_\ell}(E_k) \arrow[r, "\cong"] 
        & \ker \Phi \arrow[r, "\cong"]
        & R\bigl(\Z_{d'}\bigr)^{\eta^{d'\ell/m}}
\end{tikzcd}
\]
On the level of elements, the right-hand square is
\[
\begin{tikzcd}
    \bigl(p(\xi), p(\xi)\bigr) \arrow[r] \arrow[d, mapsto] 
        & p(\xi) \arrow[l] \arrow[d, mapsto] \\
    \bigl(p(\eta), \dotsc, p(\eta)\bigl) \arrow[r] 
        & p(\eta). \arrow[l]
\end{tikzcd}
\]
This proves the first part of the lemma.
For the second part, the relevant diagram is:
\[
\begin{tikzcd}
    K^1_{\Z_n}(E_k, \tau_\ell) \arrow[r, "\cong"] \arrow[d] 
        & \coker\left(\begin{smallmatrix} 1 & -\xi^{d\ell/n} \\ 1 & -1 \end{smallmatrix}\right) \arrow[r, "\cong"] \arrow[d]
        & R\bigl(\Z_d\bigr) / \langle 1 -\xi^{d \ell/n} \rangle \arrow[d] \\
    K^1_{\Z_m}(E_k, \tau_\ell) \arrow[r, "\cong"] 
        & \coker \Phi \arrow[r, "\cong"]
        & R\bigl(\Z_{d'}\bigr) / \langle 1 -\eta^{ d'\ell/m} \rangle 
\end{tikzcd}
\]
On the level of elements, the right-hand side square is:
\[
\begin{tikzcd}
    \bigl(p(\xi), 0 \bigr) \arrow[r] \arrow[d, mapsto] 
        & p(\xi) \arrow[l] \arrow[d, mapsto] \\
    \bigl(p(\eta), \dotsc, p(\eta), 0, \dotsc, 0 \bigl) \arrow[r] 
        & \bigl[ (1 + a + a^2 + \dotsm + a^{\frac{n d'}{m d} - 1} ) p(\eta)] \arrow[l]
\end{tikzcd}
\]
On the bottom row we have used the correspondence \eqref{cokermap}. This completes the proof.
\end{proof}

To prove that $\Z_n$-equivariant K-theory is T-admissible, we first prove an intermediary theorem that tells us that the T-duality transformation is an isomorphism for pairs with trivial bundle on one side and trivial twist on the other.
The main theorem, Theorem \ref{thm:ZnTadmiss}, is proved by reducing to this case.

\begin{theorem}\label{thm:Ek0Tiso}
The T-duality transformation for the $\Z_n$-equivariant pairs $(E_k,0)$ and $(E_0,k)$ is an isomorphism.
\end{theorem}

\begin{proof}
We prove this by reducing to the T-duality transformation for the trivial T-duality relation between $(E_0, 0)$ and itself.
This is already known to be an isomorphism; see Lemma \ref{lem:trivialtrans}. 

Let $d = \gcd(n,k)$ and consider the following diagram:
\begin{equation}\label{eq:diagE0kEk0}
\begin{tikzcd}
    K^*_{\Z_n}(E_k) \arrow[d] \arrow[r, "T"]
        & K^{*-1}_{\Z_n}(E_0, \tau_k) \arrow[d] \\
    K^*_{\Z_d}(E_k) \arrow[r] 
        & K^{*-1}_{\Z_d}(E_0, \tau_k)
\end{tikzcd}
\end{equation}
The vertical arrows are the restriction along the inclusion $\Z_d \hookrightarrow \Z_n$.
The induced $\Z_d$ action on $E_k$ is trivial, as is the twist $\tau_k$ when viewed as $\Z_d$-equivariant.
Therefore, the lower horizontal map is the T-duality transformation for the trivial T-duality triple, which is an isomorphism.

All the K-theory groups and restriction maps have been calculated in Lemma \ref{lem:EkKtheory} and Lemma \ref{lem:Ekrest}.
We show that, in this case, the restriction maps are isomorphisms.
Start by considering \eqref{eq:diagE0kEk0} with $*=0$.
We have the following identifications:
\[
\begin{tikzcd}
    K^0_{\Z_n}(E_k) \arrow[r, "\cong"] \arrow[d] 
        & R(\Z_d) \arrow[d, "\id"] \\
    K^0_{\Z_d}(E_k) \arrow[r, "\cong"] 
        & R(\Z_d)
\end{tikzcd}
\quad
\begin{tikzcd}
    K^1_{\Z_n}(E_0, k) \arrow[r, "\cong"] \arrow[d]
        & \frac{R(\Z_n)}{\langle 1 - \xi^k \rangle} \arrow[d] \\
    K^1_{\Z_d}(E_0, k) \arrow[r, "\cong"] 
        & R(\Z_d).
\end{tikzcd}
\]
The right-most vertical map is induced by $R(\Z_n) \to R(\Z_d)$, which is
surjective with kernel the ideal $\langle 1 - \xi^k\rangle$.
Hence, all of the vertical maps are isomorphisms and diagram \eqref{eq:diagE0kEk0} implies that the T-duality transformation is an isomorphism for this case.

Now consider the other case.
The identifications are as follows:
\[
\begin{tikzcd}
    K^1_{\Z_n}(E_k) \arrow[r, "\cong"] \arrow[d] 
        & R(\Z_d) \arrow[d] \\
    K^1_{\Z_d}(E_k) \arrow[r, "\cong"] 
        & R(\Z_d)
\end{tikzcd}
\quad
\begin{tikzcd}
    K^0_{\Z_n}(E_0, k) \arrow[r, "\cong"] \arrow[d]
        & R(\Z_n)^{\xi^k} \arrow[d] \\
    K^0_{\Z_d}(E_0, k) \arrow[r, "\cong"] 
        & R(\Z_d).
\end{tikzcd}
\]
In this situation, the vertical maps are seen to be injective maps onto $(n/d)R(\Z_d)$. 
The T-duality isomorphism $K^1_{\Z_d}(E_k) \cong K^0_{\Z_d}(E_0, k)$ restricts to an isomorphism of the subgroups $(n/d)K^1_{\Z_d}(E_k) \cong (n/d)K^1_{\Z_d}(E_0, k)$.
Therefore, \eqref{eq:diagE0kEk0} implies that the T-duality transformation is an isomorphism in this case as well.
\end{proof}

We are ready to prove the main theorem of this section.

\begin{theorem}\label{thm:ZnTadmiss}
$\Z_n$-equivariant K-theory is T-admissible.
\end{theorem}

\begin{proof}

Consider the $\Z_n$-equivariant T-dual pairs $(E_k, \tau_\ell)$ and $(E_\ell, \tau_k)$ and for convenience let $d = \gcd(n,k)$ and $d' = \gcd(n,\ell)$.
We must show that the corresponding T-duality transformation is an isomorphism.
The idea is to consider the following commutative diagrams:

\begin{equation}\label{diag:Tdualrest}
\begin{tikzcd}
    K^0_{\Z_n}(E_k, \tau_\ell)
        \arrow[r, "T"] \arrow[d, "rest.", swap] 
    & K^1_{\Z_n}(E_l \tau_\ell) 
        \arrow[d, "rest."] \\
    K^0_{\Z_d}(E_k, \tau_\ell)^{\Z_n} 
        \arrow[r, "T"] 
    & K^1_{\Z_d}(E_l, \tau_\ell)^{\Z_n}
\end{tikzcd}
\quad
\begin{tikzcd}
    K^1_{\Z_n}(E_k, \tau_\ell)
        \arrow[r, "T"] \arrow[d, "rest.", swap] 
    & K^0_{\Z_n}(E_l \tau_\ell) 
        \arrow[d, "rest."] \\
    K^1_{\Z_d}(E_k, \tau_\ell)^{\Z_n} 
        \arrow[r, "T"] 
    & K^0_{\Z_d}(E_l, \tau_\ell)^{\Z_n}
\end{tikzcd}
\end{equation}
The horizontal maps are T-duality transformations and the vertical maps are restrictions along the inclusion $\Z_d \hookrightarrow \Z_n$.
We will calculate that on the left-hand side the restrictions are isomorphisms and on the right-hand side they are injective maps onto the subgroups $C \cdot K^1_{\Z_d}(E_k, \tau_\ell)^{\Z_n}$ and $C \cdot K^0_{\Z_d}(E_l, \tau_\ell)^{\Z_n}$, respectively, for a fixed integer $C$.
Once we have done this, the theorem will be proved, since Theorem \ref{thm:Ek0Tiso} implies that the lower T-duality transformations are isomorphisms and, for the right-hand diagram, this remains true then restricting to the aforementioned subgroups.

Now for the computation.
By Lemma \ref{lem:EkKtheory}, the K-theory groups are
\[
K^0_{\Z_d}(E_k, \tau_\ell) \cong R(\Z_d)^{\xi^{\ell}},
\,\,\,
K^1_{\Z_d}(E_k, \tau_\ell) \cong \frac{R(\Z_d)}{\langle 1 - \xi^{\ell} \rangle}, \text{ and}
\]
\[
K^0_{\Z_d}(E_\ell, \tau_k) \cong R(\Z_{\gcd(d,\ell)}) \cong 
K^1_{\Z_d}(E_\ell, \tau_k).
\]
One should remember that the notation $E_k$ and $\tau_k$ refers to the $\Z_n$-action; as $\Z_d$-equivariant objects we would write $E_0$ and $\tau_0$.
Acting via a the generator of $\Z_n$ induces a $\Z_d$-equivariant automorphism of $E_k$ and pulling back along this map gives an automorphism of the K-theory groups.
For this, it must be noted that the pullback of $\tau_\ell$ is canonically isomorphic to $\tau_\ell$.
Using the Mayer-Vietoris sequence  -- the same one used in the proof of Lemma \ref{lem:Ekrest} -- one can calculate that the generator of $\Z_n$ acts via multiplication by $\xi^{d\ell/n} \in R(\Z_d)$  on $K^i_{\Z_d}(E_k, \tau_\ell)$ and via multiplication by $\zeta^{d'k/n} \in R(\Z_{\gcd(d, \ell)})$ on $K^i_{\Z_d}(E_\ell, \tau_k)$, where $\xi$ and $\zeta$ are generators of their respective representation rings.

We use Lemma \ref{lem:Ekrest} to write out each of the restriction maps explicitly.
The first we consider is:
\[
\begin{tikzcd}
    K^0_{\Z_n}(E_k, \tau_\ell)
        \arrow[r, "\cong"] \arrow[d]
    & R(\Z_d)^{\xi^{d\ell/n}}  
        \arrow[d] \\
    K^0_{\Z_d}(E_k, \tau_\ell)
        \arrow[r, "\cong"] 
    & R(\Z_d)^{\xi^{\ell}} 
\end{tikzcd}
\]
This restriction map is just the inclusion.
This becomes an isomorphism if we restrict the co-domain to the $\xi^{d\ell/n}$-invariant subgroup.
We have seen that this is identified with the $\Z_n$-invariant subgroup, so we have the isomorphism we desire.
The next restriction to consider is:
\[
\begin{tikzcd}
    K^1_{\Z_n}(E_\ell, \tau_k)
        \arrow[r, "\cong"] \arrow[d]
    & R(\Z_{d'}) / \bigl\langle 1 - \xi^{d'k/n} \bigr\rangle
        \arrow[d] \\
    K^1_{\Z_d}(E_\ell, \tau_k)
        \arrow[r, "\cong"]
    & R(\Z_{\gcd(d,\ell)})
\end{tikzcd}
\]
This map is
\[
    [p(\xi)] \mapsto  
    \bigl(1 + \eta^{d'k/\ell} + \dotsm + (\eta^{d'k/\ell})^{\frac{n \gcd(d,\ell)}{dd'} - 1} \bigr) p(\eta).
\]
In this case, $\frac{n \gcd(d,\ell)}{dd'}$ is the multiplicative order of $\eta^{d'k/\ell}$.
Thus, this becomes an isomorphism if we restrict to the $\eta^{d'k/\ell}$-invariant subgroup.
One sees this by noting that the map induces a bijection between the elements $1, \xi, \dotsc, \xi^{\gcd(d', d'k/n)-1}$ and the elements
\[
\eta^j \bigl(1 + \eta^{d'k/\ell} + \dotsm + (\eta^{d'k/\ell})^{\frac{n \gcd(d,\ell)}{dd'} - 1} \bigr)
\]
for $j \in \{0, \dotsc, \gcd(d', d'k/\ell) - 1 \}$, both of which form a $\Z$-linear basis for their respective groups.
We have now established that the restriction maps in the left-hand side diagram of \eqref{diag:Tdualrest} are isomorphisms.

The next restriction map is:
\[
\begin{tikzcd}
    K^1_{\Z_n}(E_k, \tau_\ell)
        \arrow[r, "\cong"] \arrow[d]
    & R(\Z_d) / \langle 1 - \xi^{d\ell/n} \rangle
        \arrow[d] \\
    K^1_{\Z_d}(E_k, \tau_\ell)
        \arrow[r, "\cong"] 
    & R(\Z_d) / \langle 1 - \eta^{\ell} \rangle
\end{tikzcd}
\]
The map is
\[
    \bigl[ p(\xi) \bigr] 
    \mapsto
    \Bigl[ \bigl(1 + \xi^{d\ell/n} + \dotsm + (\xi^{d\ell/n})^{\frac{n}{d}-1} \bigr) p(\xi) \Bigr] .
\]
Since the multiplicative order of $\xi^{d\ell/n}$ is $\alpha := \frac{\gcd(d,\ell)}{\gcd(d, \ell, d\ell/n)}$, we have
\[
    1 + \xi^{d\ell/n} + \dotsm + (\xi^{d\ell/n})^{\frac{n}{d}-1}
    = \frac{n}{d\alpha}
        \bigl( 1 + \xi^{d\ell/n} + \dotsm + (\xi^{d\ell/n})^{\alpha-1} \bigr)
\]
We see that the restriction map is an isomorphism onto the subgroup $\frac{n}{d\alpha} \cdot K^1_{\Z_d}(E_k, \tau_\ell)^{\xi^{d\ell/n}}$.

The final restriction map is:
\[
\begin{tikzcd}
    K^0_{\Z_n}(E_\ell, \tau_k)
        \arrow[r, "\cong"] \arrow[d]
    & R(\Z_{d'})^{\xi^{d'k/n}}
        \arrow[d, shift right = 3] \\
    K^0_{\Z_d}(E_\ell, \tau_k)
        \arrow[r, "\cong"] 
    & R(\Z_{\gcd(d,\ell)}) \phantom{XX}
\end{tikzcd}
\]
This map is again simply $p(\xi) \mapsto p(\eta)$.
Note that
\begin{align*}
    1 + \xi^{d'k/n} + \dotsm + (\xi^{d'k/n})^{\beta - 1} 
    &  \longmapsto 
     1 + \eta^{d'k/n} + \dotsm + (\eta^{d'k/n})^{\beta - 1} \\
    & \hspace{10mm}   =\frac{\beta}{\beta'}\bigl(1 + \eta^{d'k/n} + \dotsm + (\eta^{d'k/n})^{\beta' - 1} \bigr),
\end{align*}
where, for notational convenience, we write $\beta = d' / \gcd(d', d'k/n)$ for the multiplicative order of $\xi^{d'k/n}$ and $\beta' = \gcd(d,\ell) / \gcd(d, \ell, d'k/n)$ for the multiplicative order of $\eta^{d'k/n}$.
Thus, we can conclude that the restriction map is an injection onto $\frac{\beta}{\beta'} \cdot K^0_{\Z_d}(E_\ell, \tau_k)^{\eta^{d'k/n}}$.

Now, for our proof to work we need that $\frac{n}{d\alpha}$ and $\frac{\beta}{\beta'}$ are equal.
Fortunately, this is true:
\begin{align*}
    \frac{n \beta'}{d\alpha \beta} 
    &= \frac{n \cdot \gcd(d, \ell) \cdot \gcd(d', d'k/n) \cdot \gcd(d, \ell, d\ell/n)}{d \cdot \gcd(d, \ell, d'k/n) \cdot d' \cdot \gcd(d, \ell)} \hspace{-13mm}
    & \\
    &= \frac{n \cdot \gcd(d', d'k/n) \cdot \gcd(d, \ell, d\ell/n)}{d \cdot \gcd(d, \ell, d'k/n) \cdot d'} 
    && \text{Cancel $\gcd(d,\ell)$.} \\
    &= \frac{ \gcd(nd', d'k) \cdot \gcd(nd, nl, dl)}{d \cdot d' \cdot \gcd(nd, nl, d'k)} 
    && \text{Insert $n$ into the $\gcd$s.} \\
    &= \frac{ d' \cdot \gcd(n, k) \cdot \gcd(nd, nl, dl)}{d \cdot d' \cdot \gcd(nd, nl, d'k)} 
    && \text{Factor out $d$ and $d'$.} \\
    &= \frac{\gcd(nd, nl, dl)}{\gcd(nd, nl, d'k)} 
    && \text{Cancel $d$ and $d'$.} \\
    &= \frac{ \gcd(n^2, nk, nl, nl, kl)}{\gcd(n^2, nl, nk, k \ell)} 
    && \text{Substitute values of $d, d'$.} \\
    &= 1  && 
\end{align*}
Therefore $\frac{n}{d\alpha} = \frac{\beta}{\beta'}$; this is the constant $C$ mentioned at the beginning of the proof.
We have now shown that all the vertical maps in \eqref{diag:Tdualrest} are isomorphisms and so, as discussed at the beginning of the proof, we are done.
\end{proof}

\section{Rational Equivariant K-Theory for Finite Groups}

The cyclic group case implies that the T-duality transformation is rationally an isomorphism for all finite groups.

\begin{theorem}
\label{thm:rationalTtrans}
Let $G$ be a finite group and let $(E,P)$ and  $(\hat E, \hat P)$ be
$G$-equivariant T-dual pairs with fiber $S^1$ over a $G$-CW-complex $X$.
Then the T-duality transformation on rational twisted equivariant K-theory,
\[
    K^*_G(E, P)_\Q 
    \xrightarrow{\,\, \cong \,\,}
    K^{*-1}_G(\hat E, \hat P)_\Q,
\]
is an isomorphism, that is, $G$-equivariant K-theory is rationally T-admissible.
\end{theorem}

This theorem is proved using the decomposition theorem introduced in \cite{DSV:decomp}*{Corollary 6.2}.
The result says that there is a decomposition map
\[
    K_G(E,P)_\Q 
    \to 
    \Bigl[ \bigoplus_{g \in G} K_{\cg g}(E|_{X^g}, P|_{E|_{X^g}})_\Q \Bigr]^G
\]
that is an isomorphism onto the elements of the right-hand side satisfying the following relation:
\begin{gather*}
    \text{If $\cg h \subseteq \cg g$, then the elements in the $h$- and $g$-summand restrict} \tag{*}\label{compatibility} \\
    \text{to the same element in $K_{\cg h}(E|_{X^g}, P|_{E|_{X^g}})$}.
\end{gather*}
Thus, we will decompose the K-theory groups and use the T-duality isomorphism for the cyclic group equivariant T-duality triples that we get from the restrictions.
By Theorem \ref{thm:ZnTadmiss}, these are isomorphisms.

\begin{proof}
Consider the following diagram:
\[
\begin{tikzcd}
    K^*_G(E,P)_\Q \arrow[r] \arrow[d] 
    & K^{*-1}_G(\hat E, \hat P)_\Q \arrow[d] \\
    \Bigl[ \bigoplus_{g \in G} K^*_{\cg g}(E|_{X^g}, P|_{E|_{X^g}})_\Q \Bigr]^G \arrow[r]
    & \Bigl[ \bigoplus_{g \in G} K^{*-1}_{\cg g}(\hat E|_{X^g}, \hat P|_{E|_{X^g}})_\Q \Bigr]^G
\end{tikzcd}
\]
The upper horizontal map is the T-duality transformation.
The vertical maps are the decomposition maps; these are induced by the inclusions $E|_{X^g} \to E$.
The lower horizontal map is induced by the T-duality transformations for the restrictions to $X^g$.
Since the T-duality transformation is natural with respect to pullbacks and morphisms of twists, these restrict to the invariant subspace and the diagram commutes.
For the same reason, the map also restricts to a map between the elements satisfying \eqref{compatibility}.
By Theorem \ref{thm:ZnTadmiss}, the lower map is an isomorphism on this specified subspace, which implies that the upper map is an isomorphism.
\end{proof}

\section{The General Case: Compact Lie Groups}

We are almost ready to prove the main result, which is that the T-duality
transformation for circle bundles and twisted equivariant K-theory is an isomorphism for all compact Lie groups.
Let us give the idea of the proof. 
We will show that the T-duality transformation being an isomorphism is equivalent to a certain pullback map being injective.
There exist results that imply that a map in equivariant K-theory is injective if it is injective when restricted to all finite subgroups.
This result will be combined with the fact that the T-duality transformation is rationally an isomorphism for all finite groups to get to our main result.
As this isomorphism has only been proved rationally, it will be helpful to know that the involved groups are torsion-free:

\begin{lemma}\label{lem:twistedktheorycircle}
Let $G$ be a finite group acting on $S^1$ via a homomorphism $\varphi \colon G \to S^1$ and let $K = \ker(\varphi)$. 
Let $P$ be a $G$-equivariant twist on $S^1$.
Then there exists a $1$-dimensional representation $\xi$ of $K$ such that
\[
    K^0_G(S^1,P) \cong R(K)^\xi
    \quad \text{and} \quad
    K^1_G(S^1,P) \cong R(K) / (1- \xi)R(K).
\]
In particular, $K^*_G(S^1, P)$ is torsion-free.
\end{lemma}

\begin{proof}
Give $S^1$ the $G$-CW-structure $S^1 = (e^0 \times G/K) \cup (e^1 \times G/K)$.
The twist $P$ can be represented by a $G$-equivariant $S^1$-bundle on $G/K$,
which is equivalent to a $1$-dimensional (complex) representation of $K$.
Let $\xi$ be this representation.
The Mayer-Vietoris sequence gives
\[
    0 \to K^0_G(X,P) \to R(K)^2 \to R(K)^2 \to K^1_G(X,P) \to 0,
\]
where the middle map is $(x,y) \mapsto (x - \xi y, x-y)$.
This implies the isomorphisms claimed in the lemma.

Being the subgroup of a torsion-free module, it is clear that $R(K)^\xi$ is torsion free.
For the second group, we observe that $\xi$, being $1$-dimensional, acts via permutations on the irreducible representations of $K$, which form a basis of $R(K)$.
By partitioning the irreducible representations into $\xi$-orbits, $R(K)$ is isomorphic to a direct sum of modules of the form $\Z[\xi] / (1-\xi^d)\Z[\xi]$, where $d$ divides the multiplicative order of $\xi$.
The quotient of such a module by the sub-module generated by $(1-\xi)$ is a free abelian group of rank $1$.
This altogether implies that $R(K)/(1-\xi)R(K)$ is torsion-free.
\end{proof}

We also need a finiteness result for our twisted K-theory groups, which is
essentially (and with the same procedure) also proven in \cite{Lahtinen}*{Proof of Proposition 4}.
\begin{proposition}\label{prop:finiteness}
  Let $L$ be a compact Lie group with a closed subgroup $G$. Let $X$ be a finite
  $G$-CW complex with equivariant twist $Q$. Then $K^*_G(X,Q)$ is finitely
  generated as module over $R(L)=K^0_G(pt)$, which acts via reduction $R(L)\to R(G)$.
\end{proposition}
\begin{proof}
  We follow the proof in \cite{Lahtinen}*{Proof of Proposition 4} of the same
  result (but for more special twists).
  
  By  \cite{Segal:EKT}*{Proposition 3.2}, $R(G)$ is finitely
  generated as $R(L)$-module. Hence finite generation as $R(L)$-module is
  equivalent to finite generation as $R(G)$-module.

 By \cite{Segal:EKT}*{Corollary 3.3}, $R(G)$ is
  Noetherian and finitely generated as a ring.
  Using the
  Mayer-Vietoris sequence, we hence obtain the claimed result by induction on the number of
  cells once we prove it for all spaces of the form $X=G/H$ for $H$ a closed
  subgroup of $G$. For those spaces, the induction axiom for twists and for
  twisted K-theory  implies that we can change the group to $H$ and the space
  to $X=*$: $K^*_G(G/H;Q)\cong K^*_H(*;Q|_H)$. By \cite{FHTtwistedktheory}*{Example 1.10},  $K^*_H(*,Q|_H)$ is isomorphic to a twisted
  representation ring $R^Q(H)$. And by \cite{Lahtinen}*{Proof of Proposition 4}, this is indeed finitely generated as $R(H)$ and therefore as
  $R(L)$-module.

  By induction on the number of cells, via the Mayer-Vietoris sequence, the
  general result follows.
\end{proof}

The following is a key part of the main proof:

\begin{lemma}
\label{lem:injpullback}
Let $G$ be a compact Lie group and $\bigl( (E,Q), (\hat E, \hat Q), u)$ a $G$-equivariant T-duality triple over a point.
Let $f \colon E \times \hat E \to *$ be the constant map.
The map
\begin{equation}\label{injectivemap}
    F^*\colon K^*_G(E,Q) \to K^*_G(f^*E, F^*Q)
\end{equation}
is injective, where $F \colon f^*E \to E$ is the canonical map in the pullback square.
\end{lemma}

\begin{proof}
Let $\pi, \hat \pi, p$ and $\hat p$ be the maps denoted in the following diagram:
\begin{center}
\begin{tikzpicture}
    \node (e) at (0,0) {$E$};
    \node (eh) at (3,0) {$\hat E$};
    \node (*) at (1.5,-0.8) {$*$};
    \node (eeh) at (1.5,0.8) {$E \times_X \hat E$};

    \draw [->] (eeh.south west) -> (e) node [midway, above] {\scriptsize$p$};
    \draw [->] (eeh.south east) -> (eh) node [midway, above] {\scriptsize$\hat p$};
    \draw [->] (e.south east) -> (*) node [midway, below] {\scriptsize $\pi$};
    \draw [->] (eh.south west) -> (*) node [midway, below] {\scriptsize$\hat \pi$};
\end{tikzpicture}
\end{center}
Assume, for now, that $G$ is finite.
Since $f = \pi \circ p$, we can write \eqref{injectivemap} as the composition
\begin{equation}\label{injectivecomposition}
    K_G^*(E,Q)
    \xrightarrow{\,\, \Pi^*\,\,}
    K_G^*(\pi^*E, \Pi^*Q)
    \xrightarrow{\,\, P^*\,\,}
    K_G^*(f^*E, F^*Q),
\end{equation}
where $\Pi\colon \pi^*E \to E$ is defined as part of the pullback square.
The pullback of a principal bundle along its own map is trivial, so $\pi^*E \to E$ is the trivial bundle.
The induced map on K-theory is therefore injective, for instance by looking at the K\"unneth exact sequence for C*-algebras \cite{Blackadar}*{23.3.1}.

For the second map, observe that $P^*$ is the first map in the T-duality transformation for the pullback of $\bigl( (E, Q), (\hat E, \hat Q), u \bigr)$ along $\pi$.
By Theorem \ref{thm:rationalTtrans}, this T-duality transformation is rationally an isomorphism, so $P^*$ is rationally injective.
By Lemma \ref{lem:twistedktheorycircle}, $K_G(E,Q)$ is torsion-free and so the image of $\Pi^*$ is torsion-free.
We can then conclude that, for finite groups, the composition \eqref{injectivecomposition} is injective.

Now we show that the finite group case implies the general case. 
It was proved by McClure that if $x \in K_G(X)$ restricts to zero in $K_H(X)$ for every finite subgroup $H$ of $G$, then $x = 0$, see \cite{McClure:restriction}.
This was generalised to KK-theory by Uuye in \cite{Uuye:restriction} under the following assumptions:
\begin{itemize}
    \item $KK^H_n(A,B)$ is a finitely generated $R(G)$-module for all $n \in \Z$ and all closed subgroups $H \subseteq G$.
    \item $KK^F_n(A,B)$ is a finitely generated group for all finite subgroups $F \subseteq G$.
\end{itemize}
Under these conditions, if an element $x \in KK^G(A,B)$ restricts to zero in $KK^H(A,B)$ for all finite subgroups $H$ of $G$ then $x = 0$.

In our case, the finiteness condition is satisfied as special case of
Proposition \ref{prop:finiteness}. For $F$ finite, the relevant K-groups were calculated in Lemma \ref{lem:twistedktheorycircle}.
They are either a subgroup or a quotient of the representation ring of a finite group.

Now, let $x \in K_G(E,Q)$ be in the kernel of $K_G^*(E,Q) \to K_G^*(f^*E, F^*Q)$ and denote by $x_H \in K_H(E,Q)$ the restriction of $x$ for $H \subseteq G$.
If $H$ is finite, we have already shown that $K_H(E,Q) \to K_H(f^*E, F^*Q)$ is injective and since $x_H$ is in its kernel, we deduce that $x_H = 0$.
Therefore, Uuye's theorem implies that $x = 0$, since $x_H = 0$ for all finite subgroups $H \subseteq G$.
This completes the proof.
\end{proof}

For each T-duality triple $\bigl( (E,P), (\hat E, \hat P), u\bigr)$, we call  $\bigl( (\hat E, \hat P), (E, P), u^{-1} \bigr)$ the dual T-duality triple.
The resulting T-duality transformation 
\[
    K^*_G(\hat E, \hat P) \to K^{*-1}_G(E,P)
\]
will be called the dual T-duality transformation.

We are now prepared for the proof of the main theorem:

\begin{theorem}\label{thm:main}
Twisted $G$-equivariant K-theory is $T$-admissible when $G$ is a compact Lie group.
\end{theorem}

\begin{proof}

We use a method introduced by Bei Liu in his G\"ottingen PhD thesis \cite{BeiLiu}. 
Consider a T-duality diagram:
\begin{equation}\label{Tdiagproof}
\begin{tikzcd}[column sep={4em,between origins},row sep=1em]
    & 
    p^*P \arrow[ld] \arrow[rd] \arrow[rr, "u"] & & 
    \hat p^* \hat P \arrow[ld] \arrow[rd] & \\
    P \arrow[rd] & & 
    E \times_X \hat E \arrow[ld, "p", swap] \arrow[rd, "\hat p"] & & 
    \hat P \arrow[ld] \\
    & 
    E \arrow[rd, "\pi", swap] & & 
    \hat E \arrow[ld, "\hat \pi"] & \\
    & & 
    X & &
\end{tikzcd}
\end{equation}
Let $f = \pi \circ p = \hat \pi \circ \hat p$ be the canonical map $E \times_X \hat E \to X$. 
Pulling back along $f$ gives a T-duality diagram over $E \times_X \hat E$. 
This pulled-back diagram is the trivial T-duality diagram, that is, the
$S^1$-bundles are trivial and the twists pull back from the base (and can be
identified canonically there via $u$).
This is because the pull-back of a principal bundle along itself is trivial.
For example, $f^*E = p^*\pi^*E$ and $\pi^*E$ is the trivial principal $S^1$-bundle over $E$.
Here, we mean it is trivial as an equivariant bundle; $G$ acts trivially on the $S^1$-fiber.
We have seen in Lemma \ref{lem:trivialtrans} that the T-duality transformation for trivial T-duality triples is an isomorphism.
Moreover, when $X$ is a point, the inverse of this T-duality transformation is its dual T-duality transformation.

Let $F \colon f^*E \to E$ and $\hat F \colon f^*\hat E \to \hat E$ be the resulting maps between the $S^1$-bundles. 
By the naturality of the T-duality transformation, we have
\begin{equation}\label{eqn:proofdiag}
\begin{tikzcd}
    K_G^*(E,P)
        \rar[shift left, "T"] \dar
    & K_G^{*-1}(\hat E, \hat P)
        \dar \lar[shift left, "T"]
    \\
    K_G^* (f^*E, F^*P) 
        \rar[shift left, "T"]
    & K_G^{*-1}(f^*\hat E, \hat F^* \hat P).
        \lar[shift left, "T"]
\end{tikzcd}
\end{equation}
Let $X$ be a point; this is indeed the only case required for T-admissibility.
Then, we know that the lower T-duality transformations are inverse to each other.
Moreover, Lemma \ref{lem:injpullback} implies that the vertical maps are injective.
We therefore have that the T-duality transformations in the top row of
\eqref{eqn:proofdiag} are isomorphisms which are inverse to each other, as required.
\end{proof}

By considering the T-dual pairs introduced in Example \ref{ex:noflux}, we have the following corollary.

\begin{corollary}\label{cor:noflux}
Let $G$ be a compact Lie group, $X$ a $G$-space and $E \to X$ a $G$-equivariant principal $S^1$-bundle with Chern class $c_1 \in H^2_{G}(X)$. There is an isomorphism
\[
    K^*_G(E) \cong K^{*-1}_G(X \times S^1, P),
\]
where $G$-acts trivially on the $S^1$-factor of $X \times S^1$ and $P$ is a $S^1$-equivariant twist classified by $c_1 \in H^2_{G}(X) \hookrightarrow H^3_{G}(X \times S^1)$.
\end{corollary}

This result also includes the case where the twist comes from the group, by considering twists classified by a class in the image of $H^2_G(*) \to H^2_G(X)$.

We can also use the method of prove which shows that the T-duality
transformation is an isomorphism to identify its inverse:

\begin{corollary}\label{cor:inverse_T_is_T}
The inverse to the T-duality transformation is the dual T-duality transformation.
\end{corollary}

\begin{proof}
Revisiting the proof of Theorem \ref{thm:main}, the result follows from two intermediary results:
\begin{enumerate}
    \item The result holds for trivial T-duality triples over any $X$.
    \item The vertical maps in \eqref{eqn:proofdiag} are injective for T-duality diagrams over any $X$.
\end{enumerate}
If both of these are true, then the proof is complete by considering \eqref{eqn:proofdiag}.

For the first statement, we use the computation of Lemma \ref{lem:trivialtrans}
\[
\begin{tikzcd}
    K_G^*(X,Q) \otimes K^*(S^1) 
        \rar["{\cong}"] \dar["{\id \otimes T}"]
    & K_G^*(X \times S^1,\pr_X^*Q)  
        \dar["T"] \\
    K_G^{*}(X,Q) \otimes K^{*-1}(S^1)
        \rar["\cong"]
    & K^{*-1}_G(X \times S^1,\pr_X^*Q)
\end{tikzcd}
\]
On the left-hand side we use graded tensor products. We know that the
T-duality transform is its own inverse (and also equal to the dual T-duality
transform) for the 
basic untwisted case of the circle, so the statement for trivial T-duality
triples in the sense of Definition \ref{def:trivial} follows.

The second result is proved in Bei Liu's thesis \cite{BeiLiu}*{Lemma 3.6}.
The proof is the same as Lemma \eqref{lem:injpullback}, except now we can use
that the T-duality transformation is an isomorphism for all compact Lie
groups.
We once again write the map as the composition \eqref{injectivecomposition}.
The first map is injective because $\pi^*E$ is the trivial bundle.
The second map is injective because it is the first map in the T-duality transformation of the pullback, which we now know is an isomorphism.
Therefore the composition is injective and the proof is complete.
\end{proof}

  The proof of Corollary \ref{cor:inverse_T_is_T} also establishes that the
  T-duality transform itself is characterized axiomatically by a short list of
  natural properties as follows.

\begin{corollary}\label{cor:uniqueness}
  Let $G$ be a compact group.
  Let $T$ be a transformation which for each $G$ equivariant T-duality triple
  $\tau=((E,P),(\hat E,\hat P),u)$ with circle fibers over a space $X$ gives a homomorphism
  \begin{equation*}
   T_\tau\colon K_G^*(E,P)\to K_G^{*-1}(\hat E,\hat P)
 \end{equation*}
 that satisfies the following properties:
 \begin{enumerate}
 \item naturality for pull-backs along maps $f\colon Y\to X$
 \item normalization for trivial twists in the sense of Definition
   \ref{def:trivial}: in this case, under the K\"unneth isomorphism we require a
   commutative diagram
   \begin{equation*}
\begin{tikzcd}
    K_G^*(X,Q) \otimes K^*(S^1) 
        \rar["{\cong}"] \dar["{\id \otimes T}"]
    & K_G^*(X \times S^1,\pr_X^*Q)  
        \dar["T"] \\
    K_G^{*}(X,Q) \otimes K^{*-1}(S^1)
        \rar["\cong"]
    & K^{*-1}_G(X \times S^1,\pr_X^*Q)
\end{tikzcd}     
   \end{equation*}
 \end{enumerate}

     Then $T$ is the T-duality transformation constructed in Sections 
    \ref{sec:T_trafo} and \ref{sec:Ktheory} together with all of its properties (in
    particular, it is an isomorphism which is its own inverse).

\end{corollary}

\section{Examples}\label{sec:examples}

The T-duality isomorphism can be a useful tool for calculating the K-theory of principal $S^1$-bundles.
We finish the paper with some example calculations.
The first can be considered a worked example of the T-duality isomorphism, while the rest are a result of it. 

\begin{example}
We walk through the case where $G = S^1$ and $X$ is a point since everything can be made explicit in this case.
Let $E_k$ be $S^1$ with the $k$th power $S^1$ action, that is, with the $S^1$-action given by
\[
    S^1 \to S^1 \subseteq \operatorname{Aut}(S^1),
    \quad z \mapsto z^k.
\]
This is the same $E_k$ as Section \ref{sec:ZnTadmiss}, except the $S^1$-action in that section is restricted to a $\Z_n$-action.
Note that $E_k \to *$ is an $S^1$-equivariant principal $S^1$-bundle over a point and that all such bundles are of this form, because
\[
    H^2_{S^1}(*) = H^2(BS^1) \cong \Z,
\]
and $E_k \to *$ is the bundle classified by $k \in \Z$. 
An application of the corresponding Gysin sequence shows that $H^3_{S^1}(E_k)= \Z$ when $k=0$ and $H^3_{S^1}(E_k)=0$ otherwise.
We then have that the pairs $(E_0, P_k)$ and $(E_k, 0)$ are T-dual to each other, where $P_k$ is the twist on $E_0$ corresponding to the integer $k$.
These are all of the possible T-duality triples.
Via the Mayer-Vietoris sequence, the relevant twisted K-groups are
\[
    K^0_{S^1}(E_k) \cong R(\Z_k) \cong K^1_{S^1}(E_0, P_k)
    \quad \text{and} \quad
    K^1_{S^1}(E_k) = 0 = K^0_{S^1}(E_0, P_k).
\]
There are thus only two non-trivial T-duality transformations, 
\[
    K^1_{S^1}(E_0, P_k) \xrightarrow{\,\, T\,\,} K^0_{S^1}(E_k)
    \quad\text{and}\quad
    K^0_{S^1}(E_k) \xrightarrow{\,\, T\,\,} K^1_{S^1}(E_0, P_k).
\]
All of the maps defining these transformations are contained in the following diagram:
\[
\begin{tikzcd}[row sep = 2mm, column sep = 6mm]
    & K^1_{S^1}(E_0 \times E_k, \pr_1^*P_k)
        \arrow[r, "\cong"]
    & K^1_{S^1}(E_0 \times E_k)
        \arrow[dr, "(\pr_2)_!"]
    & \\
    K^1_{S^1}(E_0, P_k)
        \arrow[ru, "\pr_1^*"]
    &&
    & K^0_{S^1}(E_k) 
        \arrow[dl, "\pr_2^*"] \\
    & K^0_{S^1}(E_0 \times E_k, \pr_1^*P_k)
        \arrow[ul, "(\pr_1)_!"]
    & K^0_{S^1}(E_0 \times E_k)
        \arrow[l, "\cong"]
    &
\end{tikzcd}
\]
It turns out that, in this case, each of the individual maps in this diagram is an isomorphism.
The full details are in the author's PhD thesis \cite{Dove:phd}.
\end{example}

\begin{example}
Consider $S^1$ acting on $S^2$ by rotations.
By a standard Mayer-Vietoris argument, $H^2_{S^1}(S^2) \cong \Z^2$.
Let $E_{p,q}$ denote the $S^1$-equivariant principal $S^1$-bundle on $S^2$ classified by $(p,q) \in \Z^2$.
This can be explicitly constructed as
\[
    E_{p,q} = (D^2 \times S^1_p) \cup_{S^1 \times S^1_0} (D^2 \times S^1_q)
     \to D^2 \cup_{S^1} D^2 = S^2,
\]
where $S^1_k$ denotes $S^1$ with the $k$th power action.
The gluing maps come from the fact that $S^1 \times S^1_k$ is $S^1$-equivariantly homeomorphic to $S^1 \times S^1_0$ for all $k \in \Z$.

Another Mayer-Vietoris argument tells us that
\[
    H^3_{S^1}(E_{p,q})
    = 
    \begin{cases}
        \Z^2, & p=q=0,\\ 
        \Z, & p=0, q \neq 0, \\
        \Z, & p\neq0, q = 0, \\
        0,  & \text{otherwise}.
    \end{cases}
\]
If $p=q=0$, then $E_{p,q}$ is the trivial bundle.
Let $P_{k,\ell}$ be the twist on $E_{0,0}$ classified by $(k,\ell) \in \Z^2$.
Then $(E_{0,0}, P_{p,q})$ is T-dual to $(E_{p,q}, 0)$.
This is the situation described in Corollary \ref{cor:noflux}.

Let $P_k$ and $Q_\ell$ denote the twists on $E_{p,0}$ and $E_{0,q}$ classified by $k,l \in \Z$, respectively.
Then $(E_{p,0}, P_q)$ is T-dual to $(E_{0,q}, Q_p)$.
The corresponding T-duality isomorphism is
\[
    K^*_G(E_{p,0}, P_q) \cong K^{*-1}_G(E_{0,q}, Q_p).
\]
\end{example}

\begin{example}
Let $E \to X$ be a (non-equivariant) principal $S^1$-bundle on $X$ with Chern class $c_1 \in H^2(X)$.
The bundle $E^{\otimes k} = E \otimes \dotsm \otimes E$, which is classified by $c_1^k$, has a natural action of the symmetric group $\Sigma_k$.
This makes $E^{\otimes k} \to X$ a $\Sigma_k$-equivariant principal $S^1$-bundle on $X$, where $X$ is given the trivial $\Sigma_k$-action.
Corollary \ref{cor:noflux} implies that
\[
    K^*_{\Sigma_k}(E^{\otimes k}) 
    \cong K^{*-1}_{\Sigma_k}(X \times S^1, P)
    \cong K^{*-1}(X \times S^1, P) \otimes R(\Sigma_k),
\]
where $P$ is a twist classified by the image of $c_1^k$ under 
\[
    H^2(X) 
    \hookrightarrow H^3(X \times S^1) 
    \hookrightarrow H^3_{\Sigma_k}(X \times S^1).
\]
If we further assume that $K^i(X) = 0$ for $i \in \{0,1\}$, then a Mayer-Vietoris argument for $X \times S^1$ reveals that 
\[
    K^{i-1}(X \times S^1, P) \cong K^{i-1}(X)^{L^k}
    \,\, \text{and}
\]
\[
    K^i(X \times S^1, P) \cong K^{i-1}(X) / L^k K^{i-1}(X),
\]
where $L_k$ is the line bundle classified by $c_1^k$ (in other words the line bundle associated with $E^{\otimes k}$). 
So, in this case,
\[
    K^i_{\Sigma_k}(E^{\otimes k})
    \cong
    K^{i-1}(X)^{L^k} \otimes R(\Sigma_k)
    \,\, \text{and}
\]
\[
    K^{i-1}_{\Sigma_k}(E^{\otimes k})
    \cong
    \frac{K^{i-1}(X)}{L^k K^{i-1}(X)} \otimes R(\Sigma_k).
\]
\end{example}


\appendix
\section{Twisted Equivariant Cohomology}\label{appendix}

Here, we introduce an axiomatic definition of equivariant twists and twisted equivariant cohomology, following Bunke and Schick \cite{BunkeSchick05}.
The T-duality transformation and the notion of T-admissibility will be defined for any twisted equivariant cohomology theory satisfying these axioms.

With an axiomatic approach, one does not need to choose a specific model for twists.
In K-theory for example, there are many notions of twists, including principal $PU(\H)$-bundles, Hitchin gerbes, bundle gerbes, and \v Cech cocycles.
It can be difficult to translate between these directly.
Instead of choosing a specific one, we rely on a set of axioms that are sufficient to prove the results we need.
The same is true for twisted equivariant cohomology. 
Even though our main focus is K-theory, we are leaving room for other twisted cohomology theories to have a T-duality isomorphism.

\subsection{Equivariant Twists}\label{sec:equitwists}

\noindent A model for $G$-equivariant twists consists of a presheaf of monoidal groupoids 
\[
    (X,G) \mapsto \twist_G(X)
\]
on the category of spaces with group action.
This must satisfy the following properties:
\begin{enumerate}

    \item There is a natural monoidal transformation 
    \[
        \twist_G(X) \to H^3_G(X;\Z), 
        \quad P \mapsto [P],
    \]
    where the Borel cohomology group $H^3_G(X;\Z) := H^3(X \times_G EG ;\Z)$ is viewed as a monoidal category with only identity morphisms.
    This transformation classifies the isomorphism classes of $\twist_G(X)$.

    \item There is a natural ``Borel construction'', meaning there is a monoidal transformation
    \[
        \twist_G(X) \to \twist(X \times_G EG).
        \quad P \mapsto P \times_G EG
    \]
    This transformation is natural with respect to the classification by
    $H^3_G(X; \Z)$.
  \item Let $H$ be a subgroup of $G$. There must be (as part of the
    structure) a restriction transformation of presheaves $\twist_G(X)\to
    \twist_H(X)$ compatible with the restriction functor from the category of
    $G$-spaces to the category of $H$-spaces.

    \item Let $H$ is a subgroup of $G$ and $X$ a $H$-space. 
    The canonical $H$-map $X \to X \times_H G|_H$ induces a natural equivalence of groupoids
    \begin{align*}
        \twist_G(X \times_H G) \to\twist_H(X\times_H G|)\to  \twist_H(X).
    \end{align*}
    The inverse will be denoted $P \mapsto \Ind_H^G(P)$.
    
    \item (Gluing twists) Let $X = U \cup V$ where $U$ and $V$ are $G$-invariant subspaces, $P_U \in \twist_G(U)$, $P_V \in \twist_G(V)$ and $u \colon P_U|_{U \cap V} \cong P_V|_{U \cap V}$ an isomorphism.
    Then, there exists $P \in \twist_G(X)$ such that there are isomorphisms $\psi_U \colon P|_U \cong P_U$ and $\psi_V \colon P|_V \cong P_V$ satisfying $u = \psi_V \circ \psi_U^{-1}$.
    Moreover, this $P$ is unique up to isomorphism.

\end{enumerate}

The first four axioms are used to prove that T-admissibility implies that the T-duality transformation is an isomorphism.
The gluing axiom is used to prove Proposition \ref{prop:twistiso}, which allowed us to prove that $G$-equivariant pairs over $X$ are T-dual if and only if their Borel constructions are T-dual over $X \times_G EG$.

\begin{example}
Our twists of choice when working with twisted K-theory are stable equivariant $PU(\H)$-bundles.
We expand on these twists in \ref{princpuh}.
\end{example}

\begin{example}
In \cite{TuXuLG}, the authors define twisted K-theory for differentiable stacks.
The twists they use are $S^1$-central extensions over groupoids.
So, $S^1$-central extensions over groupoid representatives of the global quotient stack $X /\!\!/ G$ are a model of equivariant twists.
\end{example}

\begin{example}
A Hitchin gerbe for a space $X$ is an open cover $\{U_i\}$ together with a collection of line bundles $L_{ij} \to U_{i} \cap U_j$ and isomorphisms $\delta_{ijk} \colon L_{ij} \otimes L_{jk}  \to L_{ik}$ satisfying a cocycle condition, see \cite{Hitchin:gerbes}.
By considering equivariant line bundles, we obtain the notion of equivariant Hitchin gerbes.
These can be made quite explicit and so are useful in calculations, see for example \cite{FHTtwistedktheory}*{\textsection 1}.
\end{example}

\subsection{Axioms for Twisted Equivariant Cohomology}

Fix a model of twists $\twist_{G}(X)$.
A twisted equivariant cohomology theory consists of, for each $n \in \Z$, a functor 
\[
    (X,G,P) \mapsto h^n_G(X,P)
\]
on the category of spaces equipped with a group action and equivariant twist. 
In addition to functoriality with respect to spaces, groups, and twists, these functors must satisfy the following axioms:

\begin{enumerate}

    \item (Homotopy invariance) If $f \colon X \to Y$ and $g \colon X \to Y$
      are equivariantly homotopic, then
    \[
        g^* = u^* \circ f^* \colon h_G(Y,P) \to h_G(X,g^*P)
    \]
    for some isomorphism $u \colon f^*P \cong g^*P$.

    \item (Push-forward) If $p \colon Y \to X$ is a $G$-equivariant $h_G$-oriented map, then there is an integration map
    \[
        p_! \colon h^n_G(Y,p^*P) \to h^{n-d}_G(X, P).
    \] 
    This map has a degree shift of $d := \dim X - \dim Y$ and is natural with
    respect to pullbacks. For products with the circle, it implements the
    suspension isomorphism, compare Corollary \ref{cor:suspension}.
    \item (Induction) If $H \subseteq G$ is a subgroup, then there is a natural isomorphism
    \[
        h^n_H\bigl(E, P\bigr) 
        \cong 
        h^n_G\bigl(E \times_H G, \Ind_H^G(P)\bigr).
    \]

    \item (Mayer-Vietoris) If $X = U \cup V$ is a decomposition of $X$ into two open $G$-sets, then there is a natural long exact sequence:

\vspace{\baselineskip}
\noindent\makebox[\textwidth]{%
    \begin{tikzpicture}

        \def \rsep {2.3};
        \def \csep {-0.75};

        \node (a) at (-0.5,0)
            {$\dotsm \to h^{n-1}_G(U \cap V, P|_{U \cap V}\bigr)$};
        \node (b) at (\rsep,\csep)
            {$h^n_G(X, P)$};
        \node (c) at (2*\rsep, 0)
            {$h^n_G(U,P|_U) \oplus h^n_G(V,P|_U)$};
        \node (d) at (3*\rsep+1, \csep)
            {$h^n_G(U \cap V, P|_{U \cap V}\bigr) \to \dotsm $};
        
        \draw[->, out=320, in=180]
            ([shift={(1,0)}]a.south) to (b.west);
        \draw[->, out=0, in=230]
            (b.east) to ([shift={(-0.5,0)}]c.south);
        \draw[->, out=320, in=180]
            ([shift={(0.5,0)}]c.south) to (d.west);

    \end{tikzpicture}
}
\end{enumerate}

\begin{corollary}\label{cor:suspension}
  As usual, a consequence of Mayer-Vietoris and homotopy invariance is the
  suspension isomorphism which we formulate in the following way
  (incorporating also a condition on the push-forward maps): There is a natural
  split exact sequence
  \begin{equation*}
    0\to h^n_G(X,P)\xrightarrow{\pr_X^*} h^n_G(X\times S^1, \pr_X^*P)
    \xrightarrow{p_!} h^{n-1}(X,P)\to 0,
  \end{equation*}
  where $G$ acts trivially on the $S^1$-factor.
\end{corollary}

\begin{example}
Given a twisted non-equivariant cohomology theory, one can obtain a twisted equivariant cohomology theory via the Borel construction, that is,
\[
    h_G(X,P) := h(X \times_G EG, P).
\]
A particular example of this is twisted Borel K-theory.
\end{example}

\begin{example}
We have discussed in Section \ref{sec:Ktheory} the Thom isomorphism in twisted equivariant K-theory and hence have a push-forward along K-oriented maps. 
Therefore twisted equivariant K-theory is a twisted equivariant cohomology theory.
\end{example}

\subsection{Principal \texorpdfstring{$PU(\H)$}{PU(H)}-bundles}
\label{princpuh}

\noindent In this paper, we use stable equivariant principal $PU(\H)$-bundles in our definition of twisted equivariant K-theory.
Here, we provide the definition and their basic properties. 

\begin{definition}
Let $X$ be a $G$-space. 
A $G$-equivariant principal $PU(\H)$-bundle on $X$ is a principal $PU(\H)$-bundle $P \to X$ together with an action of $G$ on $P$ such that
\[
    \pi( p\cdot g) =  \pi(p)\cdot g
    \quad \text{and} \quad
     (p \cdot u)\cdot g = (p \cdot g) \cdot u,
\]
where $g \in G$, $p \in P$, and $u \in PU(\H)$. Note that we use right actions
for the group and for the principal bundle.
\end{definition}

\begin{definition}
A stable homomorphism $f \colon G \to PU(\H)$ is a homomorphism such that for
the induced central extension $  1\to S^1 \to \widetilde G\to G\to 1$
defined by $\widetilde G:= f^*U(\H)$ 
the induced homomorphism $\widetilde f\colon \widetilde G \to U(\H)$ contains all of the irreducible representations of $\widetilde G$ where the central $S^1$ acts via scalar multiplication, countably infinitely many times.
\end{definition}

A homomorphism $G \to PU(\H)$ is the same as a $G$-equivariant principal $PU(\H)$-bundle over a point.
A motivation for considering stable homomorphisms comes from the bijection \[
    \operatorname{Hom}_{st}\bigl( G, PU(\H) \bigr) / PU(\H)
    \longleftrightarrow
    \operatorname{Ext}(G,S^1)
\]
between stable homomorphisms up to conjugation and $S^1$-central extensions of $G$ \cite{BEJU:universaltwist}*{Prop 1.6}.
In other words, we want $G$-equivariant twists over a point to correspond to $S^1$-extensions of $G$ and this is only true if we restrict to the stable homomorphisms.

A further motivation is that the twisted $G$-equivariant K-theory of a point should be the ring of ``twisted'' representations of $G$ corresponding to the twist, that is, representations of the resulting central extension $\widetilde G$ such that the central $S^1$ acts by scalar multiplication. This is not guaranteed if the twist is not stable, see \cite{BEJU:universaltwist}*{\textsection 4.3.4} and \cite{LuckUribe}*{\textsection 15.2}.

We extend this now to the notion of stable equivariant principal $PU(\H)$-bundle.
These are, roughly speaking, $G$-equivariant principal bundles where, locally,
isotropy groups  act via stable homomorphisms.
This definition comes from \cite{BEJU:universaltwist}*{Def 2.2}.

\begin{definition}
A stable $G$-equivariant principal bundle $\pi \colon P \to X$ is one such
that for each $x \in X$ (with isotropy group $G_x$) there exists a $G$-neighbourhood $V$ of $x$ and a $G_x$-contractible slice $U$ of $x$ such that $V \cong U \times_{G_x} G$ together with a local trivialisation
\[
    P|_V \cong (PU(\H) \times U) \times_{G_x} G,
\]
where $G_x$ acts on the $PU(\H)$-factor via a stable homomorphism.
\end{definition}

Let us show that stable equivariant principal $PU(\H)$-bundles satisfy the twist axioms we have presented.
The first result states that stable equivariant principal bundles satisfy Axioms 1 and 2.

\begin{proposition}
There is a bijection between the isomorphism classes of stable $G$-equivariant principal $PU(\H)$-bundles on $X$ and $H^3_G(X;\Z)$.
This classification factors through the Borel construction
\[
    \operatorname{Proj}_G(X) 
    \to
    \operatorname{Proj}(X \times_G EG)
    \cong H^3_G(X).
\]
\end{proposition}

\begin{proof}
See \cite{AtiyahSegal:TKT}*{Prop 6.3} or, in the case that $G$ is discrete, in \cite{BEJU:universaltwist}*{Theorem 3.8}.
\end{proof}

The next results proves that stable equivariant $PU(\H)$-bundles have an induction isomorphism, giving us Axiom 3.

\begin{proposition}\label{prop:inductionPUbundles}
If $P \to X$ is a $H$-equivariant stable $PU(\H)$-principal bundle then $P \times_H G \to X \times_H G$ is a $G$-equivariant stable $PU(\H)$-principal bundle.
Therefore, the assignment $P \mapsto P \times_H G$ induces a bijection between the $H$-equivariant stable bundles on $X$ and the $G$-equivariant stable bundles on $X \times_H G$.
\end{proposition}

\begin{proof}
It is clear enough that $P \times_H G \to X \times_H G$ is a $G$-equivariant $PU(\H)$-principal bundle; we only show that it is stable.
Consider $[x,g] \in X \times_H G$.
The isotropy group at $[x,g]$ is $H_x$ and this does not depend on the choice of $x$ in the coset $[x,g]$.
We must prove that there exists an open neighbourhood $V$ of $[x,g]$ and a slice $U$ such that 
\[
    V \cong U \times_{H_x} G 
    \quad \text{and} \quad
    (P \times_H G)|_V \cong (PU(\H) \times U) \times_{H_x} G, 
\]
where $H_x$ acts on $PU(\H)$ via a stable homomorphism.
Since $P \to X$ is stable, there exists an open neighbourhood $V'$ of $x$ and a slice $U'$ such that
\[
    V' \cong U' \times_{H_x} H
    \quad \text{and} \quad 
    P|_{V'} \cong (PU(\H) \times U') \times_{H_x} H.
\]
Let $V = V' \times_H G$ and let $U$ be the image of $U$ in $X \times_H G$ under the injection $y \mapsto [y,g]$.
Then $U$ is a slice at $[x,g]$ with
\[
    V 
    = V' \times_H G 
    \cong (U' \times_{H_x} H) \times_H G 
    \cong U' \times_{H_x} G 
    \cong U \times_{H_x} G, 
\]
and so
\begin{align*}
    (P \times_H G)|_V 
    &= P|_{V'} \times_H G  \\
    &\cong (PU(\H) \times U') \times_{H_x} H \times_H G  \\
    &\cong (PU(\H) \times U') \times_{H_x} G  \\
    &\cong (PU(\H) \times U ) \times_{H_x} G.
\end{align*}
Therefore, $P \times_H G$ is a $G$-equivariant stable principal $PU(\H)$-bundle.
\end{proof}

For Axiom 4, we use that bundles can be glued together and that stability is a local property, so that the gluing of two stable equivariant bundles is again stable.
The outcome of this discussion is the following:

\begin{proposition}
Stable equivariant principal $PU(\H)$-bundles satisfy the twist axioms introduced in Section \ref{sec:equitwists}.
\end{proposition}


\section{Declarations}

The authors declare that all the data supporting the findings of this study are
available within the paper.

The authors have no competing interests to declare that are relevant to the content of this article.
All authors certify that they have no affiliations with or involvement in any organization or entity with any financial interest or non-financial interest in the subject matter or materials discussed in this manuscript.
The authors have no financial or proprietary interests in any material
discussed in this article.

Indirectly, the authors hope to gain reputation in the scientific community
via the publication of this article. This might influence decisions of
current or future employers which have financial implications for the authors.


\end{document}